\documentclass{article}
\usepackage[utf8]{inputenc}
\usepackage{fullpage}
\usepackage[ngerman,english]{babel}
\usepackage{fontawesome}
\usepackage{hyperref}
\usepackage{graphicx}

\usepackage{mathtools}
\usepackage{amsmath}
\usepackage{amsthm}
\usepackage{amsrefs}
\usepackage{amsfonts}
\usepackage{bbm}
\usepackage{csquotes}

\usepackage{tikz}

\usepackage{xcolor}
\definecolor{alert}{HTML}{CD5C5C}
\definecolor{myblue}{HTML}{616AC5}
\usepackage{nicefrac}

\usepackage{tcolorbox}
\tcbuselibrary{most}

 \tcbset{tcbox width=auto,left=1mm,top=1mm,bottom=1mm,
right=1mm,boxsep=1mm,middle=1pt}

\usepackage[utf8]{inputenc}
\usepackage[english]{babel}
 
\newtheorem{theorem}{Theorem}[section]
\newtheorem{corollary}{Corollary}[theorem]
\newtheorem{lemma}[theorem]{Lemma}
\newtheorem{proposition}[theorem]{Proposition}

\theoremstyle{definition}
\newtheorem{definition}[theorem]{Definition}

\theoremstyle{definition}
\newtheorem{example}[theorem]{Example}

\theoremstyle{remark}
\newtheorem*{remark}{\textbf{Remark}}

\theoremstyle{section}
\newtheorem*{ttheorem}{Theorem \ref{fundamental_Q}}

\theoremstyle{section}
\newtheorem*{tttheorem}{Theorem \ref{weighted_composition_general}}

\theoremstyle{section}
\newtheorem*{ttttheorem}{Theorem \ref{GKZ_general}}

\theoremstyle{section}
\newtheorem*{tttttheorem}{Theorem \ref{CPO_Hardy_full}}

\title{\textbf{Cyclicity preserving operators \\ on spaces of analytic functions in $\mathbb{C}^n$}}
\date{}
\author{Jeet Sampat}

\begin{document}

\maketitle \begin{abstract}

For spaces of analytic functions defined on an open set in $\mathbb{C}^n$ that satisfy certain nice properties, we show that operators that preserve shift-cyclic functions are necessarily weighted composition operators. Examples of spaces for which this result holds true consist of the Hardy space $H^p(\mathbb{D}^n) \, (0 < p < \infty)$, the Drury-Arveson space $\mathcal{H}^2_n$, and the Dirichlet-type space $\mathcal{D}_{\alpha} \, (\alpha \in \mathbb{R})$. We focus on the Hardy spaces and show that when $1 \leq p < \infty$, the converse is also true. The techniques used to prove the main result also enable us to prove a version of the Gleason-Kahane-\.Zelazko theorem for partially multiplicative linear functionals on spaces of analytic functions in more than one variable.

\end{abstract}

\section{Introduction}

The results presented here are motivated by a number of questions about cyclic functions in the Hardy space $H^p(\mathbb{D}^n)$. Recall that $\mathbb{D}^n := \{ z \in \mathbb{C}^n \, \big \rvert \, |z_i| < 1, \, \forall \, 1 \leq i \leq n \}$ is the unit polydisc in $\mathbb{C}^n$, and for $0 < p < \infty$ $$H^{p}(\mathbb{D}^{n}) := \bigg \{ f \in \text{Hol}(\mathbb{D}^{n})\; \bigg \rvert \; ||f||_{p}^{p} := \underset{0 \leq r < 1}{\text{sup}} \int \limits_{\mathbb{T}^{n}} |f(rw)|^{p} \; d\sigma_{n}(w) < \infty \bigg \}$$ where $\sigma_n$ denotes the normalized Lebesgue measure on the $n$-torus $\mathbb{T}^n := \{ z \in \mathbb{C}^{n}\; \big \rvert \; |z_{i}| = 1, \; \forall \; 1 \leq i \leq n \}$. It is known that $H^p(\mathbb{D}^n)$ is a Banach space for all $1 \leq p < \infty$ with norm $||\cdot||_p$. A function $f \in H^p(\mathbb{D}^n)$ is said to be (shift) cyclic if $$\text{S}[f] := \overline{\text{span}}\{ z^{\alpha}f(z) \; | \; \alpha \in \mathbb{Z}^{+}(n) \} = \overline{\text{span}}\{ pf \; | \; p \text{ - polynomial} \} = H^p(\mathbb{D}^n)$$ where $\mathbb{Z}^+(n)$ is the set of all $n$-tuples $\alpha = (\alpha_i)_{i = 1}^n$ of non-negative integers, and $z^{\alpha} := z_1^{\alpha_1} z_2^{\alpha_2} \dots z_n^{\alpha_n}$. We also have the space $H^{\infty}(\mathbb{D}^n)$, $$ H^{\infty}(\mathbb{D}^n) := \Big \{ f \in \text{Hol}(\mathbb{D}^n) \, \Big \rvert \, ||f||_{\infty} := \underset{w \in \mathbb{D}^n}{\text{sup}} |f(w)| < \infty \Big \} $$ which is the set of all bounded analytic functions defined on $\mathbb{D}^n$. Just like $H^p(\mathbb{D}^n)$ for $1 \leq p < \infty$, $H^{\infty}(\mathbb{D}^n)$ is a Banach space with the supremum norm $||\cdot||_{\infty}$.

\vspace{1.5mm}

For $H^p(\mathbb{D})$ when $0 < p < \infty$, cyclic functions have been characterized using Beurling's theorem and the canonical factorization theorem (see \emph{Theorem 7.4} in \cite{duren1970theory} and \emph{Theorem 4} in \cite{10.2307/1994442}). In the case when $n > 1$, we do not have a version of Beurling's theorem nor the canonical factorization theorem (see \emph{Section 4.4} in \cite{rudin1969function} for more details). Several sufficient conditions for cyclicity in $H^p(\mathbb{D}^n)$ were provided by N. Nikolski (\emph{Theorems 3.3} and \emph{3.4}, \cite{AIF_2012__62_5_1601_0}), but there are no non-trivial necessary conditions known as of now.

\vspace{1.5mm}

One way of obtaining necessary conditions would be through operators that preserve cyclicity. That is, identify all operators $T : H^p(\mathbb{D}^n) \xrightarrow{} H^q(\mathbb{D}^m)$ such that $Tf$ is cyclic whenever $f$ is cyclic.

\vspace{1.5mm}

\noindent In the case when $p = q = 2$ and $n = m = 1$, a result of P. C. Gibson, M. P. Lamoureux and G. F. Margrave shows that all such operators have to be weighted composition operators and vice versa (\emph{Theorem 4}, \cite{gibson2015constructive}). This was generalized to $H^p(\mathbb{D})$ for $0 < p \leq \infty$ by J. Mashreghi and T. Ransford with the following theorem (\emph{Theorem 2.2}, \cite{mashreghi2015gleason}).

\begin{theorem} \label{mash_Hardy}

Let $0 < p \leq \infty$ and let $T : H^p(\mathbb{D}) \xrightarrow{} \text{\emph{Hol}}(\mathbb{D})$ be a linear map such that $Tg(z) \neq 0$ for all outer functions (same as cyclic functions when $0 < p < \infty$) $g \in H^p(\mathbb{D})$ and all $z \in \mathbb{D}$. Then there exist holomorphic maps $\phi : \mathbb{D} \xrightarrow{} \mathbb{D}$ and $\psi : \mathbb{D} \xrightarrow{} \mathbb{C} \setminus \{0\}$ such that \[ Tf = \psi \cdot (f \circ \phi) \; \big (\, \forall \, f \in H^p(\mathbb{D}) \big ). \]

\end{theorem}

It is important to note that continuity of $T$ is not assumed in the above theorem. Also, the theorem uses outer functions instead of cyclic functions but the notion of outer functions coincides with cyclic functions in the case $n = 1$ and $0 < p < \infty$. See \emph{Definition \ref{definition_outer}} for a brief discussion about outer functions. For more general spaces over $\mathbb{D}$, J. Mashreghi and T. Ransford prove the following theorem (\emph{Theorem 3.2}, \cite{mashreghi2015gleason}).

\begin{theorem} \label{mash_general}

Suppose $X \subset \text{\emph{Hol}}(\mathbb{D})$ satisfies $(X1)$-$(X3)$ and $Y \subset X$ satisfies $(Y1)$-$(Y2)$ as defined below. Let $T : X \xrightarrow{} \text{\emph{Hol}}(\mathbb{D})$ be a continuous linear map such that $Tg(z) \neq 0$ for every $g \in Y$ and $z \in \mathbb{D}$. Then there exist holomorphic functions $\phi : \mathbb{D} \xrightarrow{} \mathbb{D}$ and $\psi : \mathbb{D} \xrightarrow{} \mathbb{C} \setminus \{0\}$ such that $$Tf(z) = \psi(z) f(\phi(z))$$ for each $f \in X$.

\end{theorem}

Properties $(X1)$-$(X3)$ are:

\begin{enumerate}

\item[$(X1)$] $X$ contains the set of polynomials, and they form a dense subspace of $X$.

\item[$(X2)$] For each $w \in \mathbb{D}$, the evaluation map $f \mapsto f(w) : X \xrightarrow{} \mathbb{C}$ is continuous.

\item[$(X3)$] $X$ is shift-invariant, i.e. $f \in X \Rightarrow zf \in X$.

\end{enumerate}

and properties $(Y1)$-$(Y2)$ are:

\begin{enumerate}

\item[$(Y1)$] If $g \in X$ and $0 < \text{inf}_{\mathbb{D}}|g| \leq \text{sup}_{\mathbb{D}}|g| < \infty$, then $g \in Y$.

\item[$(Y2)$] If $g(z) = z - \lambda$ where $\lambda \in \mathbb{T}$, then $g \in Y$.

\end{enumerate}

\noindent The proof of \emph{Theorem \ref{mash_general}} relies on classifying $\Lambda \in X^*$ such that $\Lambda(g) \neq 0, \forall g \in Y$ (\emph{Theorem 3.1}, \cite{mashreghi2015gleason}). This is similar to a result known as the Gleason-Kahane-\.Zelazko (GK\.Z) theorem (see \cite{gleason1967characterization} and \cite{kahane1968characterization}), which identifies multiplicative linear functionals in a complex unital Banach algebra through its action on invertible elements (see \emph{Theorem \ref{GKZ}} in \emph{Section 5} below). In \cite{mashreghi2015gleason}, it is shown that a version of the GK\.Z theorem holds for modules of a complex unital Banach algebra and it can be applied to the multiplier algebra of the space $X$ satisfying properties $(X1)$-$(X3)$ above to obtain \emph{Theorem \ref{mash_general}}.

\vspace{1.5mm}

In \cite{kou2017linear}, K. Kou and J. Liu provide a similar but simpler argument for $H^p(\mathbb{D})$ when $ 1 < p < \infty$, which is essentially the same as that of \emph{Theorem \ref{mash_general}},  but instead of the subset $Y$ they consider the set $\{ e^{w \cdot z} \, | \, w \in \mathbb{C} \}$ (see \emph{Theorem 2}, \cite{kou2017linear}). They also showed that the converse of \emph{Theorem \ref{mash_general}} is true when $1 < p < \infty$, i.e. all weighted composition operators on $H^p(\mathbb{D})$ for $1 < p < \infty$ also preserve outer (and thus cyclic) functions.

\vspace{1.5mm}

Using techniques similar to those in \cite{kou2017linear} and \cite{mashreghi2015gleason}, we can generalize \emph{Theorem \ref{mash_general}} to spaces of analytic functions in more that one variable and also over arbitrary domains in a simpler way. The fundamental result presented here is the following theorem from \emph{Section 3}.

\begin{ttheorem}

Suppose $\mathcal{X}$ satisfies \textbf{Q1}-\textbf{Q3} over a set $\hat{D} \subset \mathbb{C}^n$ for some $n \in \mathbb{N}$. Let $\Lambda \in \mathcal{X}^*$ be such that $$\Lambda(e^{w \cdot z}) \neq 0 \text{ for every } w \in \mathbb{C}^n.$$ Then, there exist $a \in \mathbb{C} \setminus \{0\}$ and $b \in \hat{D}$ such that $\Lambda(f) = a \cdot f(b)$.

\end{ttheorem}

\noindent Here, properties \textbf{Q1}-\textbf{Q3} are modified versions of $(X1)$-$(X3)$ (see \emph{Section 3}), so that they make sense for spaces of analytic functions in more than one variable. In \emph{Section 4}, using \emph{Theorem \ref{fundamental_Q}}, we will obtain the following generalization of \emph{Theorem \ref{mash_general}} mentioned above.

\begin{tttheorem}

Suppose $\mathcal{X}$ is a space of functions on a set $\hat{D} \subset \mathbb{C}^n$ that are analytic in an open set $D \subset \hat{D}$, and satisfies \textbf{Q1}-\textbf{Q3}. Suppose $\mathcal{Y}$ is a topological vector space of functions on a set $E$ such that $\Gamma_u g := g(u)$, $g \in \mathcal{Y}$ defines a continuous linear functional for all $u \in E$. Let $T : \mathcal{X} \xrightarrow{} \mathcal{Y}$ be a continuous linear operator. Then, the following are equivalent :

\begin{enumerate}

\item[$(1)$] $T(e^{w \cdot z})$ is non-vanishing for every $w \in \mathbb{C}^n$.

\item[$(2)$] $Tf(u) = a(u) f(b(u))$ for some $a \in \mathcal{Y}$ non-vanishing, and a map $b : E \xrightarrow{} \hat{D}$.

\end{enumerate}

In this case, note that $a = T(1)$ and $b = \frac{T(z)}{T(1)}$, where $T(z)$ represents the $n$-tuple of functions $\big (T(z_i) \big )_{i = 1}^n$.

\end{tttheorem}

Using \emph{Theorem \ref{weighted_composition_general}} and some other facts about Hardy spaces in several complex variables, we prove the following generalization of \emph{Theorem \ref{mash_Hardy}} above and \emph{Theorem 2} in \cite{kou2017linear}.

\begin{tttttheorem}

$(1)$ Fix $0 < p,q < \infty$ and $m, n \in \mathbb{N}$. Let $T : H^p(\mathbb{D}^n) \xrightarrow{} H^q(\mathbb{D}^m)$ be a bounded linear operator such that $Tf$ is cyclic whenever $f$ is cyclic. Then, $T$ is necessarily a weighted composition operator, i.e. there exist analytic functions $a \in H^q(\mathbb{D}^m)$ and $b : \mathbb{D}^m \xrightarrow{} \mathbb{D}^n$ such that $Tf(z) = a(z) f(b(z))$ for every $z \in \mathbb{D}^m$ and $f \in H^p(\mathbb{D}^n)$.

\vspace{1.5mm}

Furthermore $a = T1$ is cyclic and $b = \frac{T(z)}{T1}$, where $T(z) = \big (T(z_i) \big )_{i = 1}^n$.

\vspace{1.5mm}

\noindent In the case when $1 \leq q < \infty$, the converse is also true. That is, all bounded weighted composition operators from $H^p(\mathbb{D}^n)$ into $H^q(\mathbb{D}^m)$ also preserve cyclicity.

\vspace{1.5mm}

$(2)$ Fix $0 < p,q \leq \infty$ and $m,n \in \mathbb{N}$. Let $T : H^{p}(\mathbb{D}^n) \xrightarrow{} H^{q}(\mathbb{D}^m)$ be a bounded linear operator such that $Tf$ is outer whenever $f$ is outer. Then, $T$ is necessarily a weighted composition operator, i.e. there exist analytic functions $a \in H^q(\mathbb{D}^m)$ and $b : \mathbb{D}^m \xrightarrow{} \mathbb{D}^n$ such that $Tf(z) = a(z) f(b(z))$ for every $z \in \mathbb{D}^m$ and $f \in H^{p}(\mathbb{D}^n)$.

\vspace{1.5mm}

Furthermore, $a = T1$ is outer and $b = \frac{T(z)}{T1}$, where $T(z) = \big (T(z_i) \big )_{i = 1}^n$.

\end{tttttheorem}

\noindent Here, $f \in H^p(\mathbb{D}^n)$ is said to be outer if $$\log|f(0)| = \frac{1}{2 \pi} \int_{\mathbb{T}^n} \log |f|.$$ See the discussion after \emph{Definition \ref{definition_outer}} below for more details on outer functions.

\vspace{1.5mm}

As mentioned before, there is a version of the GK\.Z theorem for modules of complex unital Banach algebras in \cite{mashreghi2015gleason}. As an interesting byproduct of the main results presented here, we will prove the following result for Banach spaces of analytic functions, instead of modules of Banach algebras, classifying all partially multiplicative linear functionals in terms of their action on a certain set of exponentials.

\begin{ttttheorem}

Suppose $\widehat{\mathcal{X}}$ is a space of analytic functions on a set $\widehat{D} \subset \mathbb{C}^n$ that satisfies \textbf{Q1}-\textbf{Q3}. Let $\Lambda \in \mathcal{X}^*$ such that $\Lambda(1) = 1$. Then, the following are equivalent :

\begin{enumerate}

\item[$(i)$] $\Lambda(e^{w \cdot z}) \neq 0$ for every $w \in \mathbb{C}^n$.

\item[$(ii)$] $\Lambda \equiv \Lambda_b$ for some $b \in \widehat{D}$.

\item[$(iii)$] $\Lambda$ is \textbf{M2}.

\item[$(iv)$] $\Lambda$ is \textbf{M1}.

\end{enumerate}

\end{ttttheorem}

\subsection{Acknowledgements}

This paper would not have been possible without the guidance of my PhD advisor, Greg Knese.

\vspace{1.5mm}

\noindent I would also like to thank John McCarthy, Brett Wick, and my colleagues Alberto Dayan and Christopher Felder for valuable discussions on topics covered in this paper and mathematics in general.

\vspace{1.5mm}

\noindent This research was partially supported by the NSF grant DMS-1900816.

\vspace{1.5mm}

Let us start by setting up some helpful notation.

\section{Notations and preliminary results}

Fix $n \in \mathbb{N}$. For an open set $D \subset \mathbb{C}^n$, let $\mathcal{X} \subset \text{Hol}(D)$ be a Banach space satisfying the following properties :

\begin{enumerate}

\item[\textbf{P1}] The set of polynomials $\mathcal{P}$ is dense in $\mathcal{X}$.

\item[\textbf{P2}] The point-evaluation map $\Lambda_z : \mathcal{X} \xrightarrow{} \mathbb{C}$, defined as $\Lambda_z(f) := f(z)$ for every $f \in \mathcal{X}$, is a bounded linear functional on $\mathcal{X}$ for every $z \in D$.

\item[\textbf{P3}] The $i^{th}$-shift operator $S_i : \mathcal{X} \xrightarrow{} \mathcal{X}$, defined as $S_i f(z) := z_i f(z)$ for every $(z_k)_{k=1}^n = z \in D$ and $f \in \mathcal{X}$, is a bounded linear operator for every $1 \leq i \leq n$.

\end{enumerate}

Examples of spaces that satisfy \textbf{P1}-\textbf{P3} include the Hardy space $H^{p}(\mathbb{D}^{n})$ for $1 \leq p < \infty$, the Drury-Arveson space $\mathcal{H}^2_n$ on the unit ball $\mathbb{B}_n := \big \{ z \in \mathbb{C}^n \, \big \rvert \, \underset{i = 1}{\overset{n}{\sum}} |z_i|^2 < 1 \big \}$, and the Dirichlet-type spaces $\mathcal{D}_{\alpha}$ for $\alpha \in \mathbb{R}$. $$ \mathcal{H}^2_n := \bigg \{ f \in \text{Hol}(\mathbb{B}_n) \; \Big \rvert \; ||f||^2_{\mathcal{H}^2_n} := \underset{a \in \mathbb{Z}^+(n)}{\sum} \frac{a_1 ! \; a_2 ! \, \dots \, a_n !}{(a_1 + a_2 + \dots + a_n)!} |\hat{f}(a)|^2 < \infty \bigg \} $$ $$\mathcal{D}_{\alpha} := \bigg \{ f \in \text{Hol}(\mathbb{D}^n) \; \Big \rvert \; ||f||_{\mathcal{D}_\alpha}^{2} := \underset{a \in \mathbb{Z}^{+}(n)}{\sum} \big ((a_1+1) \dots (a_n+1)\big )^{\alpha}|\hat{f}(a)|^2 < \infty \bigg \} $$ The list of Dirichlet-type spaces consists of many important spaces like the usual Dirichlet space ($\alpha = 1$), the Hardy space $H^2(\mathbb{D}^n)$ ($\alpha = 0$), and also the Bergman space ($\alpha = -1$). For these spaces, we prove the following preliminary result.

\begin{theorem} \label{fundamental_P}

Suppose $\mathcal{X}$ satisfies \textbf{P1}-\textbf{P3} over an open set $D \subset \mathbb{C}^n$. Let $\Lambda \in \mathcal{X}^*$ be such that $\Lambda(e^{w \cdot z}) \neq 0$ for every $w \in \mathbb{C}^n$. Then, there exist $a \in \mathbb{C} \setminus \{0\}$ and $b \in \sigma_{r}(S)$ such that $\Lambda p = a \cdot p(b)$ for every $p \in \mathcal{P}$. Here, $\sigma_{r}(S)$ is the right Harte spectrum of $S = (S_i)_{i = 1}^n$.

\end{theorem}

Recall that $\sigma_{r}(S)$ is the complement in $\mathbb{C}^n$ of $\rho_{r}(S)$, where \[ \rho_{r}(S) := \Big \{ \lambda \in \mathbb{C}^n \; \Big \rvert \; \; \exists \; \{A_i\}_{i = 1}^n \subset \mathcal{B}(\mathcal{X}) \text{ such that } \sum_{i = 1}^n (S_i - \lambda_i I)A_i = I \Big \} . \]

\noindent Note that it is not immediate from \textbf{P1}-\textbf{P3} that $e^{w \cdot z} \in \mathcal{X}$. We address this separately as a lemma before we prove \emph{Theorem \ref{fundamental_P}}.

\begin{lemma} \label{exponential_existence}

For each $w \in \mathbb{C}^n$, we have $e^{w \cdot z} \in \mathcal{X}$. In fact, $p_k := \underset{|\alpha| \leq k}{\sum} \frac{w^{\alpha} \cdot z^{\alpha}}{\alpha!} \xrightarrow{} e^{w \cdot z}$ in $\mathcal{X}$ as $k \xrightarrow{} \infty$, where $|\alpha| := \alpha_1 + \dots + \alpha_n$ and $\alpha! := \alpha_1! \, \alpha_2! \, \dots \, \alpha_n! \,$.

\end{lemma}

\begin{proof}

Fix $w \in \mathbb{C}^n$. We show that $\underset{k \xrightarrow{} \infty}{\lim} p_k$ exists. This follows from the fact that $\mathcal{X}$ is a Banach space and
$$ \underset{\alpha \in \mathbb{Z}^+(n)}{\sum} \Big \rvert \Big \rvert \frac{w^{\alpha} \cdot z^{\alpha}}{\alpha !} \Big \rvert \Big \rvert \leq \underset{\alpha \in \mathbb{Z}^+(n)}{\sum} \frac{ |w|^{\alpha}||S||^{\alpha}||1||}{\alpha !} = ||1|| e^{|w| \cdot ||S||} \, \text{; } |w| = \big (|w_1|,\dots,|w_n| \big ) \text{, } ||S|| = \big (||S_1||,\dots,||S_n|| \big ) $$ Let $g = \underset{k \xrightarrow{} \infty}{\lim} p_n$ in $\mathcal{X}$. Note that $p_k$ converges to $e^{w \cdot z}$ point-wise. By \textbf{P2}, this implies $g(z) = e^{w \cdot z}$. \qedhere

\end{proof}

\begin{proof}[Proof of Theorem \ref{fundamental_P}]

Given that $\Lambda(e^{z \cdot w}) \neq 0$ for every $w \in \mathbb{C}^n$. As $\Lambda$ is continuous, we get $$ \underset{\alpha \in \mathbb{Z}^{+}(n)}{\sum} \frac{\Lambda(z^{\alpha}) \cdotp w^{\alpha}}{\alpha!} \neq 0, \, \forall w \in \mathbb{C}^{n}.$$ Let $\lambda_{\alpha} := \Lambda(z^{\alpha})$, $\forall \alpha \in \mathbb{Z}^{+}(n)$. Now, $|\lambda_{\alpha}| \leq ||\Lambda(z^{\alpha})|| \leq ||\Lambda||\cdot||z^{\alpha}||$. Thus, $$|\lambda_{\alpha}| \leq ||\Lambda||\cdot||S_{1}||^{\alpha_{1}}\cdot||S_{2}||^{\alpha_{2}}\cdot\cdot\cdot\cdot||S_{n}||^{\alpha_{n}}\cdot||1|| \text{ for every } \alpha \in \mathbb{Z}^{+}(n).$$

If we define $F(w) := \underset{\alpha \in \mathbb{Z}^{+}(n)}{\sum} \frac{\lambda_{\alpha}w^{\alpha}}{\alpha!}$, $\forall w \in \mathbb{C}^{n}$, then $F$ is a non-vanishing entire function such that $|F(w)| \leq ||\Lambda||\cdot||1||\cdot e^{|w| \cdot ||S||}$. So, there exist $a_{0} \in \mathbb{C}$ and $b \in \mathbb{C}^{n}$ such that $F(w) = e^{a_{0} + b \cdot w}$ (see \emph{Theorem \ref{non_vanishing_entire}}). Using the definition of $F(w)$ above and comparing power-series coefficients, we get $\lambda_{\alpha} = e^{a_{0}}b^{\alpha}$, $\forall \alpha \in \mathbb{Z}^{+}(n)$. Let $a := e^{a_{0}} \in \mathbb{C} \setminus \{ 0 \}$. This means $$ \Lambda(z^{\alpha}) = a \cdotp b^{\alpha}, \, \forall \alpha \in \mathbb{Z}^{+}(n).$$ Note that we have shown $\Lambda p = a \cdotp p(b)$ for every polynomial $p$. It only remains to show that $b \in \sigma_{r}(S)$.

\vspace{1.5mm}

\noindent For the sake of contradiction, suppose $b \not \in \sigma_{r}(S)$. Thus, there exists $\{A_i\}_{i = 1}^n \subset \mathcal{B}(\mathcal{X})$ such that $\sum_{i = 1}^n (S_i - \lambda_i I)A_i = I$. In particular, $\sum_{i = 1}^n (z_i - \lambda_i)A_i1 = 1$.

\vspace{1.5mm}

\noindent Fix $\epsilon > 0$ and pick $p_i \in \mathcal{P}$ for each $1 \leq i \leq n$ such that $|| A_i1 - p_i || < \epsilon \big / (n \cdot ||\Lambda|| \cdot ||S_i - \lambda_i I||)$. Note that, $$\bigg | \bigg | 1 - \sum_{i = 1}^n (z_i - \lambda_i)p_i \bigg | \bigg | = \bigg | \bigg | \sum_{i = 1}^n (z_i - \lambda_i)(A_i1-p_i) \bigg | \bigg | \leq \sum_{i = 1}^n || S_i - \lambda_i I || \cdot || A_i 1 - p_i || < \sum_{i = 1}^n \frac{\epsilon}{n \cdot ||\Lambda||} \leq \frac{\epsilon}{||\Lambda||}.$$

\noindent Based on the representation of $\Lambda$ on polynomials, we know that $\Lambda \big (\sum_{i = 1}^n (z_i - \lambda_i)p_i \big ) = 0$. This means $$|a| = | \Lambda1 | = \bigg | \Lambda1 - \Lambda \bigg (\sum_{i = 1}^n (z_i - \lambda_i)p_i \bigg ) \bigg | \leq ||\Lambda|| \cdot \bigg | \bigg | 1 - \sum_{i = 1}^n (z_i - \lambda_i)p_i \bigg | \bigg | < \epsilon.$$ As $\epsilon > 0$ was arbitrarily chosen and $a$ is non-zero, we get a contradiction. Hence, $b \in \sigma_{r}(S)$. \qedhere

\end{proof}

\begin{remark}

It would be great if we can show that $b \in D$ but that need not be the case. It is obvious that $D \subset \sigma_r(S)$, but it may not be possible to extend the domain of every function in $\mathcal{X}$ to the whole of $\sigma_r(S)$ in order to extend the functional in the theorem to all of $\mathcal{X}$ as a point-evaluation. We shall use the notion of \emph{maximal domains} to address this issue later in \emph{Section 3}. For that, we mention the last part of the proof of \emph{Theorem \ref{fundamental_P}} as a separate result.

\end{remark}

\begin{proposition} \label{spectrum_lemma}

Suppose $\mathcal{X}$ satisfies \textbf{P1}-\textbf{P3} over an open set $D \subset \mathbb{C}^n$. Let $b \in \mathbb{C}^n$ such that $\Lambda_b(p) := p(b)$ for every $p \in \mathcal{P}$ has a bounded linear extension to all of $\mathcal{X}$. Then $b$ lies in $\sigma_r(S)$.

\end{proposition}

\begin{example} \label{Hardy_envelope_1}

In the case when $\mathcal{X} = H^p(\mathbb{D}^n)$, for some $1 \leq p < \infty$, it is easy to check that $\sigma_r(S) = \overline{\mathbb{D}^n}$. So, $b$ obtained in \emph{Theorem \ref{fundamental_P}} lies in $\overline{\mathbb{D}^n}$. We claim that in this case, $b$ lies in $\mathbb{D}^n$. For the sake of argument, assume that $b = (b_i)_{i = 1}^n \in \partial \mathbb{D}^n$ with $b_j \in \mathbb{T}$ for some $1 \leq j \leq n$. Consider $q(z) :=  z_j - b_j$. Since $z - \beta$ is cyclic in $H^p(\mathbb{D})$ for all $1 \leq p < \infty$ and $\beta \not \in \mathbb{D}$, $q$ is a cyclic polynomial in $H^p(\mathbb{D}^n)$ for all values of $p$. This means that for any given $f \in H^p(\mathbb{D}^n)$, there exist polynomials $\{q_k\}_{k \in \mathbb{N}}$ such that $q_k q \xrightarrow{} f$. Note that $$\Lambda(q_k q) = a \cdot q_k(b)q(b) = 0$$ for every $k \in \mathbb{N}$ since $q(b) = 0$. Thus, $\Lambda(f) = 0$ which implies $\Lambda \equiv 0$, a contradiction. So, $b \in \mathbb{D}^n$ in this case and $\Lambda \equiv a \Lambda_b$.

\end{example}

We can make a similar argument for spaces $\mathcal{X}$ that have an \emph{envelope of cyclic polynomials over $D$}. Recall that $f \in \mathcal{X}$ is cyclic if $$\text{S}[f] := \overline{\text{span}} \{ z^{\alpha}f(z) \; | \; \alpha \in \mathbb{Z}^{+}(n) \} = \overline{\text{span}} \{ pf \; | \; p \in \mathcal{P} \} = \mathcal{X}.$$ Equivalently, $f \in \mathcal{X}$ is cyclic if and only if $1 \in \text{S}[f]$, since polynomials form a dense subspace of $\mathcal{X}$. It is also easy to see that all cyclic functions are non-vanishing.

\begin{definition} \label{definition_envelope}

We say that $\mathcal{X}$ has an envelope of cyclic polynomials over $D$ if there is a family $\mathcal{F} \subset \mathcal{P}$ of cyclic polynomials such that $$\underset{q \in \mathcal{F}}{\bigcap} \big ( \mathbb{C}^{n} \setminus \mathcal{Z}(q) \big ) \subseteq D$$ where $\mathcal{Z}(q)$ is the zero-set of $q$.

\end{definition}

\begin{proposition}

Suppose $\mathcal{X}$ satisfies \textbf{P1}-\textbf{P3} over an open set $D \subset \mathbb{C}^n$, and also has an envelope $\mathcal{F} \subset \mathcal{P}$ of cyclic polynomials. Let $\Lambda \in \mathcal{X}^*$ be such that $\Lambda(e^{w \cdot z}) \neq 0$ for every $w \in \mathbb{C}^n$. Then there exist $a \in \mathbb{C} \setminus \{ 0 \}$ and $b \in D$ such that $\Lambda f = a \cdot f(b)$ for all $f \in \mathcal{X}$.

\end{proposition}

\begin{proof}

Getting $a \in \mathbb{C} \setminus \{ 0 \}$ and $b \in \mathbb{C}^n$ such that $\Lambda(p) = a \cdot p(b)$ for every $p \in \mathcal{P}$ is the same as that in \emph{Theorem \ref{fundamental_P}}. We only need to show that $b \in D$ since in that case, $\Lambda \equiv a \Lambda_b$ on $\mathcal{X}$. For this, let $q \in \mathcal{F}$ be arbitrary and suppose $q(b) = 0$. Since $q$ is cyclic, for every $f \in \mathcal{X}$ we obtain a sequence of polynomials $\{q_k\}_{k \in \mathbb{N}}$ such that $q_k q \xrightarrow{} f$. This means that $$\Lambda(q_k q) = a \cdot q_k(b) q(b) = 0 \xrightarrow{} \Lambda(f)$$ which implies $\Lambda \equiv 0$. A contradiction. Therefore $q(b) \neq 0$ for every $q \in \mathcal{F}$ and thus, $$b \in \underset{q \in \mathcal{F}}{\bigcap} \big ( \mathbb{C}^{n} \setminus \mathcal{Z}(q) \big ) \subseteq D$$ as required. \qedhere

\end{proof}

\begin{example} \label{example_Hardy_envelope}

In the case of Hardy spaces $H^p(\mathbb{D}^n)$ for $1 \leq p < \infty$, we saw in \emph{Example \ref{Hardy_envelope_1}} that the family $\mathcal{F} := \{ z_i - \beta \; \big \rvert \; 1 \leq i \leq n \text{ and } \beta \not \in \mathbb{D} \}$ is an envelope of cyclic polynomials over $\mathbb{D}^n$. The same set of polynomials also works for the Dirichlet-type spaces $\mathcal{D}_{\alpha}$ whenever $\alpha \leq 1$.

\vspace{1.5mm}

\noindent For $\alpha > 1$ and $1 \leq i \leq n$, the polynomial $z_i - w$ with $w \in \mathbb{T}$ is not cyclic in $\mathcal{D}_{\alpha}$, and hence the same example does not work. In fact, every $f \in \mathcal{D_{\alpha}}$ is continuous up to the boundary when $\alpha > 1$.

\vspace{1.5mm}

\noindent In fact $\Lambda_b$ is a bounded linear functional on $\mathcal{D}_{\alpha}$ even when $b \in \partial \mathbb{D}^n$. Thus $\mathcal{D}_{\alpha}$ cannot have an envelope of cyclic polynomials over $\mathbb{D}^n$. Detailed discussion on cyclicity of polynomials in the Dirichlet-type spaces can be found in \cite{beneteau2016cyclic}.

\vspace{1.5mm}

\noindent Even in the case of $\mathcal{D}_{\alpha}$ when $\alpha > 1$, note that the converse of \emph{Theorem \ref{fundamental_P}} is true if we consider all $b \in \overline{\mathbb{D}^n}$. However, it is not obvious at all if it is the case for every space $\mathcal{X}$, since $b \in \sigma_r(S)$ could be arbitrary. We examine this behaviour with the help of maximal domains.

\end{example}

\section{Maximal domains}

Let us try to make sense of how big the domain of functions in a general space $\mathcal{X}$ that satisfies properties \textbf{P1}-\textbf{P3} can become without losing the structure we need.

\begin{definition}

Given $\mathcal{X}$ satisfying \textbf{P1-P3} over an open set $D \subset \mathbb{C}^n$, we define the maximal domain of functions in $\mathcal{X}$ to be the set $\hat{D} := \big \{ w \in \mathbb{C}^n \; \big \rvert \; \Lambda_w p := p(w), \; \forall p \in \mathcal{P} \text{ has a bounded linear extension to } \mathcal{X} \big \}$.

\end{definition}

\noindent By \textbf{P2}, we know that $D \subset \hat{D}$. By \emph{Proposition \ref{spectrum_lemma}}, we know that $\hat{D} \subset \sigma_r(S)$. In the case of $H^p(\mathbb{D}^n)$ for $1 \leq p < \infty$ and $\mathcal{D}_{\alpha}$ for $\alpha \leq 1$, we saw in \emph{Example \ref{example_Hardy_envelope}} that $\hat{D} = \mathbb{D}^n$. However for $\mathcal{D}_{\alpha}$ when $\alpha > 1$, $\hat{D} = \overline{\mathbb{D}^n}$.

\vspace{1.5mm}

We shall now show that in general, $\mathcal{X}$ can be identified with a space $\hat{\mathcal{X}}$ of functions over $\hat{D}$ which also satisfies some suitable modifications of properties \textbf{P1-P3}. To be precise, we will show that $\mathcal{X}$ is isometrically isomorphic to a space $\hat{\mathcal{X}}$ that satisfies \textbf{Q1}-\textbf{Q3} over $\hat{D}$ defined as follows :

\begin{enumerate}

\item[\textbf{Q1}] The set of polynomials $\mathcal{P}$ is dense in $\hat{\mathcal{X}}$.

\item[\textbf{Q2}] The point-evaluation map $\Lambda_z : \hat{\mathcal{X}} \xrightarrow{} \mathbb{C}$, defined as $\Lambda_z f := f(z)$ for every $f \in \hat{\mathcal{X}}$ is a bounded linear functional on $\hat{\mathcal{X}}$ for every $z \in \hat{D}$. Furthermore, if for some $z \in \mathbb{C}^n$ the map $\Lambda_z p := p(z)$ extends to a bounded linear functional on all of $\hat{\mathcal{X}}$, then $z \in \hat{D}$.

\item[\textbf{Q3}] The $i^{th}$-shift operator $S_i : \hat{\mathcal{X}} \xrightarrow{} \hat{\mathcal{X}}$, defined as $S_i f(z) := z_i f(z)$ for every $(z_i)_{i=1}^n = z \in \hat{D}$ and $f \in \hat{\mathcal{X}}$, is bounded for every $1 \leq i \leq n$.

\end{enumerate}

\noindent \textbf{Q2} represents the maximality of $\hat{D}$ w.r.t bounded linear extensions of point-evaluations on polynomials.

\vspace{1.5mm}

We begin with some notation before proving the identification. For every $f \in \mathcal{X}$, define $\hat{f}(\hat{z}) := \Lambda_{\hat{z}} f$ for every $\hat{z} \in \hat{D}$ where, with the abuse of notation, we write $\Lambda_{\hat{z}} f$ to represent the extension of $\Lambda_{\hat{z}}|_\mathcal{P}$ on $\mathcal{X}$ evaluated at $f$. Notice that for $z \in D$, $\hat{f}(z) = f(z)$ for every $f \in \mathcal{X}$. This also implies that $\hat{f}|_D \in \text{Hol}(D)$.

\vspace{1.5mm}

\noindent For polynomials $\hat{p} \in \hat{\mathcal{P}}$, we have that $\hat{p}(\hat{z}) = p(\hat{z})$ for every $\hat{z} \in \hat{D}$. Note that $\hat{\mathcal{P}}$ is the same set as $\mathcal{P}$ but we use them differently depending on the space in which they are being considered.

\vspace{1.5mm}

Now, let $\hat{X} := \{ \hat{f} : \hat{D} \xrightarrow{} \mathbb{C} \; | \, f \in \mathcal{X} \} $ and endow it with the natural vector space structure of point-wise addition and scalar multiplication. This can be done because it is obvious that $\hat{f} + \hat{g} = \widehat{f+g}$, and $\alpha \hat{f} = \widehat{\alpha f}$ for every $\alpha \in \mathbb{C}$, $f,g \in \mathcal{X}$.

\vspace{1.5mm}

\noindent Define the map $\iota : \mathcal{X} \xrightarrow{} \hat{\mathcal{X}}$ as $\iota(f) := \hat{f}$ for every $f \in \mathcal{X}$. $\iota$ is clearly a vector space isomorphism, and we can define $||\hat{f}||_\mathcal{\hat{X}} := ||f||_\mathcal{X}$ for every $\hat{f} \in \hat{\mathcal{X}}$. This implies $f_k \xrightarrow{} f$ in $\mathcal{X}$ if and only if $\hat{f}_k \xrightarrow{} \hat{f}$ in $\hat{\mathcal{X}}$.

\vspace{1.5mm}

\noindent So, $\hat{\mathcal{X}}$ turns into a Banach space, and $\iota$ becomes an isometric isomorphism of Banach spaces. Note that since $\hat{\mathcal{X}}|_D := \big \{ \hat{f}|_D \, \big \rvert \, \hat{f} \in \hat{\mathcal{X}} \big \} = \mathcal{X}$, we can say that $\hat{\mathcal{X}}$ is an extension of $\mathcal{X}$ to $\hat{D}$.

\vspace{1.5mm}

We now show that $\hat{\mathcal{X}}$ satisfies properties \textbf{Q1}-\textbf{Q3}. We start with properties \textbf{Q1} and \textbf{Q2}.

\begin{proposition} \label{x_hat_properties_1&2}

The space $\hat{\mathcal{X}}$ defined above satisfies \textbf{Q1} and \textbf{Q2}.

\end{proposition}

\begin{proof}

In order to show \textbf{Q1}, first recall that $f_k \xrightarrow{} f$ in $\mathcal{X}$ if and only if $\hat{f}_k \xrightarrow{} \hat{f}$ in $\hat{\mathcal{X}}$. Since $\mathcal{P}$ is dense in $\mathcal{X}$ by \textbf{P1}, it implies easily that the set of polynomials $\hat{\mathcal{P}}$ is dense in $\hat{\mathcal{X}}$.

\vspace{1.5mm}

\noindent In order to show \textbf{Q2}, notice that the map $\Lambda_{\hat{z}} \hat{f} := \hat{f}(\hat{z})$ is bounded for every $\hat{z} \in \hat{D}$ since $$|\Lambda_{\hat{z}} \hat{f}| = |\hat{f}(\hat{z})| = |\Lambda_{\hat{z}}f| \leq ||\Lambda_{\hat{z}}||_{\mathcal{X}^*}||f|| = ||\Lambda_{\hat{z}}||_{\mathcal{X}^*}||\hat{f}||.$$

\noindent For the second part of \textbf{Q2}, suppose for some $\hat{z} \in \mathbb{C}^n$, $\Lambda_{\hat{z}}$ defined as above extends to all of $\hat{\mathcal{X}}$. As $\mathcal{P}$ and $\hat{\mathcal{P}}$ are identical, we can evaluate $\Lambda_{\hat{z}}$ on polynomials in $\mathcal{P}$ to get $$|\Lambda_{\hat{z}} p| = |p(\hat{z})| = |\hat{p}(\hat{z})| \leq ||\Lambda_{\hat{z}}||_{\hat{\mathcal{X}}^*} ||\hat{p}|| \leq ||\Lambda_{\hat{z}}||_{\hat{\mathcal{X}}^*} ||p||$$ where $||p||$ in the last inequality is the norm in $\mathcal{X}$. As $\mathcal{P}$ is dense in $\mathcal{X}$, $\Lambda_{\hat{z}}$ extends to a bounded linear functional on all of $\mathcal{X}$. By construction of $\hat{D}$, this means that $\hat{z} \in \hat{D}$. \qedhere

\end{proof}

\noindent Instead of showing that $\hat{\mathcal{X}}$ satisfies \textbf{Q3} directly, we will prove a general result about multipliers which will be useful later when we consider partially multiplicative linear functionals. Recall that $\phi \in H^{\infty}(D)$ is a multiplier in $\mathcal{X}$ if $\phi f \in \mathcal{X}$ for every $f \in \mathcal{X}$. As $1 \in \mathcal{X}$, we get that $\mathcal{M}(\mathcal{X}) \subset \mathcal{X}$.

\vspace{1.5mm}

\noindent Equivalently, $\phi \in H^{\infty}(D)$ is a multiplier if multiplication by $\phi$, i.e. $M_{\phi} : \mathcal{X} \xrightarrow{} \mathcal{X}$ defined as $M_{\phi} f = \phi f$ for every $f \in \mathcal{X}$, is a bounded linear operator on $\mathcal{X}$. The set of all multipliers in $\mathcal{X}$ forms a sub-algebra of $H^{\infty}(D)$, and is denoted by $\mathcal{M}(\mathcal{X})$. With this notation, we have the following result.

\begin{proposition}

For every $\phi \in \mathcal{M}(\mathcal{X})$, $\hat{\phi}$ is a multiplier in $\hat{\mathcal{X}}$. Conversely, every multiplier $\hat{\phi}$ in $\hat{\mathcal{X}}$ can be identified with some $\phi \in \mathcal{M}(\mathcal{X})$.

\end{proposition}

\begin{proof}

First, notice that for every choice of polynomials $p$ and $q$, we have that $\widehat{pq} = \hat{p} \hat{q}$. Let $f \in \mathcal{X}$ be arbitrary, and let $\{q_k\}_{k \in \mathbb{N}}$ be a sequence of polynomials that converges to $f$ in $\mathcal{X}$. Then for every $\hat{z} \in \hat{D}$ $$\widehat{pf}(\hat{z}) = \underset{k \xrightarrow{} \infty}{\lim} \widehat{p q_k}(\hat{z}) = \underset{k \xrightarrow{} \infty}{\lim} \hat{p}(\hat{z}) \hat{q}_k(\hat{z}) = \hat{p}(\hat{z}) \underset{k \xrightarrow{} \infty}{\lim} \hat{q}_k(\hat{z}) = \hat{p}(\hat{z}) \hat{f}(\hat{z}).$$ The first equality uses the fact that $p q_k \xrightarrow{} p f \Rightarrow \widehat{pq_k} \xrightarrow{} \widehat{pf}$. Thus $\widehat{pf} = \hat{p} \hat{f}$ for every $p \in \mathcal{P}$, $f \in \mathcal{X}$. This implies $\hat{p}$ is a multiplier in $\hat{\mathcal{X}}$ since $$||\hat{p} \hat{f}|| = ||\widehat{pf}|| = ||pf|| \leq ||M_p|| \, ||f|| = ||M_p|| \, ||\hat{f}||.$$

\vspace{1.5mm}

Suppose now that $\phi$ is a multiplier in $\mathcal{X}$. We already know $\widehat{\phi q} = \hat{\phi} \hat{q}$ for every $q \in \mathcal{P}$. Let $f \in \mathcal{X}$ be arbitrary, and suppose again that $q_k \xrightarrow{} f$ for some polynomials $q_k$. It is now easy to see for every $\hat{z} \in \hat{D}$, $$\widehat{\phi f}(\hat{z}) = \underset{k \xrightarrow{} \infty}{\lim} \widehat{\phi q_k}(\hat{z}) = \underset{k \xrightarrow{} \infty}{\lim} \hat{\phi}(\hat{z}) \hat{q}_k(\hat{z}) = \hat{\phi}(\hat{z}) \underset{k \xrightarrow{} \infty}{\lim} \hat{q}_k(\hat{z}) = \hat{\phi}(\hat{z}) \hat{f}(\hat{z}).$$ Therefore $\widehat{\phi f} = \hat{\phi} \hat{f}$ for every $\phi \in \mathcal{M}(\mathcal{X})$, $f \in \mathcal{X}$ and so, $$||\hat{\phi} \hat{f}|| = ||\widehat{\phi f}|| = ||\phi f|| \leq ||M_\phi|| \, ||f|| = ||M_\phi|| \, ||\hat{f}||.$$ This proves that $\hat{\phi} \in \mathcal{M}(\hat{\mathcal{X}})$ whenever $\phi \in \mathcal{M}(\mathcal{X})$.

\vspace{1.5mm}

\noindent The converse is easy since $\hat{\phi} \hat{f} \in \hat{\mathcal{X}}$ implies there exists $g \in \mathcal{X}$ such that $\hat\phi \hat{f} = \hat{g}$. This means $g = \hat{g} \rvert_D = \phi f$ and so, $\phi f \in \mathcal{X}$. Thus $\phi \in \mathcal{M}(\mathcal{X})$ whenever $\hat{\phi} \in \mathcal{M}(\hat{\mathcal{X}})$. \qedhere

\end{proof}

\begin{corollary}

The space $\hat{\mathcal{X}}$ defined above satisfies \textbf{Q3}

\end{corollary}

\begin{proof}

As the shift operators are just multiplication operators, \emph{Proposition 3.3} shows that $\hat{\mathcal{X}}$ satisfies \textbf{Q3}. \qedhere

\end{proof}

\noindent Now that the shift operators are bounded, we can talk about cyclic functions in $\hat{\mathcal{X}}$. However, the way we have defined the norm in $\hat{\mathcal{X}}$, it is obvious that $f \in \mathcal{X}$ is cyclic if and only if $\hat{f} \in \hat{\mathcal{X}}$ is cyclic. This and the propositions above prove the following identification theorem.

\begin{theorem}

Given a space $\mathcal{X}$ that satisfies \textbf{P1}-\textbf{P3} over an open set $D \subset \mathbb{C}^n$, there exists a space $\hat{\mathcal{X}}$, consisting of functions defined over the maximal domain $\hat{D} \supseteq D$ of functions in $\mathcal{X}$, that satisfies \textbf{Q1}-\textbf{Q3} and is isometrically isomorphic to $\mathcal{X}$ with the map $\iota(f) := \hat{f},$ for $f \in \mathcal{X}$.

\vspace{1.5mm}

Furthermore $\hat{\mathcal{X}}|_D := \big \{ \hat{f}|_D \, \big \rvert \, \hat{f} \in \hat{\mathcal{X}} \big \} = \mathcal{X}$, and $\hat{\mathcal{X}}$ has the same set of multipliers and cyclic functions as $\mathcal{X}$. That is, $\phi \in \mathcal{M}(\mathcal{X})$ if and only if $\hat{\phi} \in \mathcal{M}(\hat{\mathcal{X}})$, and $f$ is cyclic in $\mathcal{X}$ if and only if $\hat{f}$ is cyclic in $\hat{\mathcal{X}}$.

\end{theorem}

\noindent With the help of \emph{Theorem 3.4}, we can modify \emph{Theorem \ref{fundamental_P}} as follows.

\begin{theorem} \label{fundamental_Q}

Suppose $\mathcal{X}$ satisfies \textbf{Q1}-\textbf{Q3} over a set $\hat{D} \subset \mathbb{C}^n$ for some $n \in \mathbb{N}$. Let $\Lambda \in \mathcal{X}^*$ be such that $\Lambda(e^{w \cdot z}) \neq 0$ for every $w \in \mathbb{C}^n$. Then, there exist $a \in \mathbb{C} \setminus \{0\}$ and $b \in \hat{D}$ such that $\Lambda \equiv a \Lambda_b$.

\end{theorem}

\begin{proof}

The proof of this theorem is the same as that of \emph{Theorem \ref{fundamental_P}} except by property \textbf{Q2}, we directly obtain $b \in \hat{D}$ instead of having to show that $b \in \sigma_r(S)$. \qedhere

\end{proof}

It should be noted that while \emph{Theorem \ref{fundamental_Q}} is not a better result compared to \emph{Theorem \ref{fundamental_P}}, it shows that the point $b$ is not completely arbitrary; in the sense that functions in $\mathcal{X}$ are well-behaved around $b$ and that most of the structure we need can be extended to it.

\vspace{1.5mm}

\noindent In special cases like $\mathcal{D}_{\alpha}$ when $\alpha > 1$, since the maximal domain is known to be different from the domain we started with, we get a better understanding of where the point $b$ lies and that it makes sense to talk about the evaluation at the point $b$ for any function in the space.

\vspace{1.5mm}

We have now finally covered all the preliminaries required to obtain cyclicity preserving operators and also prove a version of the GK\.Z theorem for these spaces.

\section{Cyclicity preserving operators}

First, we prove a generalization of \emph{Theorem \ref{mash_general}} that follows easily from the previous results.

\begin{theorem} \label{weighted_composition_general}

Suppose $\mathcal{X}$ is a space of functions on a set $\hat{D} \subset \mathbb{C}^n$ that are analytic in an open set $D \subset \hat{D}$, and satisfies \textbf{Q1}-\textbf{Q3}. Suppose $\mathcal{Y}$ is a topological vector space of functions on a set $E$ such that $\Gamma_u g := g(u)$, $g \in \mathcal{Y}$ defines a continuous linear functional for all $u \in E$. Let $T : \mathcal{X} \xrightarrow{} \mathcal{Y}$ be a continuous linear operator. Then, the following are equivalent :

\begin{enumerate}

\item[$(1)$] $T(e^{w \cdot z})$ is non-vanishing for every $w \in \mathbb{C}^n$.

\item[$(2)$] $Tf(u) = a(u) f(b(u))$ for some $a \in \mathcal{Y}$ non-vanishing, and a map $b : E \xrightarrow{} \hat{D}$.

\end{enumerate}

In this case, note that $a = T(1)$ and $b = \frac{T(z)}{T(1)}$, where $T(z)$ represents the $n$-tuple of functions $\big (T(z_i) \big )_{i = 1}^n$.

\end{theorem}

\begin{proof}

$(2) \Rightarrow (1)$ is obvious.

\vspace{1.5mm}

\noindent Suppose now that $(1)$ holds. Fix $u \in E$ and define $\Lambda := \Gamma_u \circ T \in \mathcal{X}^*$. Note that for every $w \in \mathbb{C}^n$, we get $$\Lambda(e^{w \cdot z}) = \Gamma_u(T(e^{w \cdot z})) = T(e^{w \cdot z})(u) \neq 0$$ since $T(e^{w \cdot z})$ is non-vanishing by $(1)$. By \emph{Theorem \ref{fundamental_Q}}, we get that for all $f \in \mathcal{X}$, $\Lambda f = a(u) f(b(u))$ for some $a(u) \in \mathbb{C}$ non-zero, and $b(u) \in \hat{D}$.

\vspace{1.5mm}

\noindent Since the choice of $u \in E$ was arbitrary, we get the desired functions $a = T(1) \in \mathcal{Y}$, $b = \frac{T(z)}{T(1)} : E \xrightarrow{} \hat{D}$, and that $Tf(u) = a(u)f(b(u))$ for every $u \in E$. \qedhere

\end{proof}

We are now ready to identify cyclicity preserving operators. The only thing we require is the following.

\begin{lemma}

$e^{w \cdot z}$ is cyclic in $\mathcal{X}$ for every $w \in \mathbb{C}^n$.

\end{lemma}

\begin{proof}

Fix $w \in \mathbb{C}^n$. We need to find polynomials $p_k$ so that $||p_k e^{w \cdot z} - 1|| \xrightarrow{} 0$ as $k \xrightarrow{} \infty$. For this, let $p_k$ be the truncations of the power-series of $e^{-w \cdot z}$. We saw before in \emph{Lemma \ref{exponential_existence}} that $p_k$ converges to $e^{-w \cdot z}$ in the norm of $\mathcal{X}$.

\vspace{1.5mm}

\noindent Next, we show that $e^{w \cdot z}$ is a multiplier. Let $q_k$ be the truncations of the power-series of $e^{w \cdot z}$. Given any $f \in \mathcal{X}$, we need to show $e^{w \cdot z} f$ lies in $\mathcal{X}$. Note that by the triangle inequality, we get for every $k \leq l$ that \[ ||q_l f - q_k f|| \leq \bigg ( \displaystyle \sum_{k < |\alpha| \leq l} \frac{|w|^{\alpha} \cdot ||S||^{\alpha} \cdot ||1||}{\alpha !} \bigg ) ||f||. \] Therefore $q_k f$ is a Cauchy sequence and thus, converges to some function $g \in \mathcal{X}$. As $q_k \xrightarrow{} e^{w \cdot z}$ point-wise, by \textbf{P2} we get that $q_k f \xrightarrow{} e^{w \cdot z} f$ which implies $e^{w \cdot z} \in \mathcal{M}(\mathcal{X})$. This means that $M_{e^{w \cdot z}}(e^{-w \cdot z}) = 1$ and thus, $M_{e^{w \cdot z}}(p_k) \xrightarrow{} 1$ as $k \xrightarrow{} \infty$. That is, $p_k e^{w \cdot z} \xrightarrow{} 1$ as $k \xrightarrow{} \infty$. \qedhere

\end{proof}

\noindent With this in mind, the following is a trivial consequence of \emph{Theorem \ref{weighted_composition_general}}.

\begin{theorem} \label{CPO_general_Q}

Suppose $\mathcal{X}$ is a space of functions defined on a set $\hat{D} \subset \mathbb{C}^n$ that are analytic in an open set $D \subset \hat{D}$, and satisfies \textbf{Q1}-\textbf{Q3}. Suppose $\mathcal{Y}$ is a space of analytic functions on an open set $E \subset \mathbb{C}^m$ that satisfies \textbf{P2} and \textbf{P3}. Let $T : \mathcal{X} \xrightarrow{} \mathcal{Y}$ be such that $Tf$ is cyclic whenever $f$ is cyclic. Then, there exist analytic functions $a \in \mathcal{Y}$ and $b : E \xrightarrow{} \hat{D}$ such that $Tf(u) = a(u)f(b(u))$ for every $u \in E$.

\vspace{1.5mm}

Moreover, $a = T(1)$ is cyclic and $b = \frac{T(z)}{T(1)}$ where $T(z)$ is the $n$-tuple of functions $\big (T(z_i) \big )_{i = 1}^n$.

\end{theorem}

\noindent To see how this translates to spaces $\mathcal{X}$ that satisfy properties \textbf{P1}-\textbf{P3} instead, we get the following immediate consequence of \emph{Theorem \ref{CPO_general_Q}}.

\begin{theorem} \label{CPO_general_P}

Suppose $\mathcal{X}$ is a space of functions on an open set $D \subset \mathbb{C}^n$ that satisfies \textbf{P1}-\textbf{P3}. Suppose $\mathcal{Y}$ is a space of analytic functions on an open set $E \subset \mathbb{C}^m$ that satisfies \textbf{P2} and \textbf{P3}. Let $T : \mathcal{X} \xrightarrow{} \mathcal{Y}$ be such that $Tf$ is cyclic whenever $f$ is cyclic. Then, there exist analytic functions $a \in \mathcal{Y}$ and $b : E \xrightarrow{} \hat{D}$ such that $Tf(u) = a(u)f(b(u))$ for every $u \in E$.

\vspace{1.5mm}

Here, $\hat{D}$ is the maximal domain for functions in $\mathcal{X}$ and with the abuse of notation, $f(b(u)) = \Lambda_{b(u)} \hat{f}$. Moreover, $a = T(1)$ is cyclic and $b = \frac{T(z)}{T(1)}$ where $T(z)$ is the $n$-tuple of functions $\big (T(z_i) \big )_{i = 1}^n$.

\end{theorem}

\noindent One can immediately observe in \emph{Theorems \ref{weighted_composition_general}, \ref{CPO_general_Q}} and \emph{\ref{CPO_general_P}} that the spaces $\mathcal{X}$ and $\mathcal{Y}$ may be defined for functions in different number of variables.

\vspace{1.5mm}

\noindent Note that for \emph{Theorems \ref{CPO_general_Q}} and \emph{\ref{CPO_general_P}}, we do not get a proper equivalence easily as in \emph{Theorem \ref{weighted_composition_general}} since it is not at all trivial to determine when a weighted composition operator preserves cyclicity.

\vspace{1.5mm}

In the case when $\mathcal{X} = H^p(\mathbb{D}^n)$ and $\mathcal{Y} = H^q(\mathbb{D}^m)$ for some $1 \leq p < \infty$, we get that all operators that preserve cyclicity are necessarily weighted composition operators. The same is true for the Dirichlet-type spaces $\mathcal{D}_{\alpha}$ when $\alpha \leq 1$. For $\mathcal{X} = \mathcal{D}_{\alpha}$ when $\alpha > 1$, we need to consider the space over its maximal domain $\hat{D} = \overline{\mathbb{D}^n}$. Then, all operators that preserve cyclicity will necessarily be weighted composition operators.

\vspace{1.5mm}

We now discuss the case of Hardy spaces in detail.

\subsection{Cyclicity preserving operators on Hardy spaces}

\noindent From \emph{Theorem 2} in \cite{kou2017linear}, we know that in the case of $H^p(\mathbb{D})$ for $1 < p < \infty$, the converse of \emph{Theorem \ref{CPO_general_Q}} is true. That is, all bounded weighted composition operators on $H^p(\mathbb{D})$ also preserve cyclicity. In the case when $n > 1$, it is not obvious if all bounded weighted composition operators preserve cyclicity. We will show that the converse of \emph{Theorem \ref{CPO_general_Q}} is true whenever $\mathcal{Y} = H^q(\mathbb{D}^m)$ for some $1 \leq q < \infty$.  We need the following important properties of S$[f]$, the shift-invariant subspace generated by a function $f \in H^p(\mathbb{D}^n)$.

\begin{lemma}

Let $f \in H^p(\mathbb{D}^n)$ for some $1 \leq p < \infty$. Then, $\phi f \in \text{\emph{S}}[f]$ for each $\phi \in H^{\infty}(\mathbb{D}^n)$.

\end{lemma}

\begin{proof}

For the sake of contradiction, let $\phi f \not \in \text{S}[f]$. By the Hahn-Banach theorem, there exists $\Gamma \in (H^p(\mathbb{D}^n))^*$ such that $\Gamma(\phi f) \neq 0$ and $\Gamma |_{\text{S}[f]} \equiv 0$. Since $H^p(\mathbb{D}^n) \subset L^p(\mathbb{T}^n)$ is a closed subspace, by duality of $L^p(\mathbb{T}^n)$ there exists $h \in L^{p'}(\mathbb{T}^n)$ such that $\Gamma(g) = \int_{\mathbb{T}^n} g \overline{h}$ for every $g \in H^p(\mathbb{D}^n)$, where $p'$ is the exponent dual to $p$ (see \emph{Theorem 7.1} in \cite{duren1970theory} for more details).

Since $\phi$ is the $\text{weak-}^*$ limit of some sequence of analytic polynomials $p_k$ (take Fej\'er means for example) and $f \overline{h} \in L^1(\mathbb{T}^n)$ for each $f \in H^p(\mathbb{D}^n)$, we get that $$\Gamma(\phi f) = \int_{\mathbb{T}^n} \phi f \overline{h} = \underset{k \xrightarrow{} \infty}{\lim} \int_{\mathbb{T}^n} p_k f \overline{h} = 0.$$ The last equality follows from the fact that $p_k f \in \text{S}[f]$ for each $k \in \mathbb{N}$, and that $\int_{\mathbb{T}^n} g \overline{h} = \Gamma(g) = 0$ for every $g \in \text{S}[f]$. Thus we reach a contradiction since $\Gamma$ was chosen so that $\Gamma(\phi f) \neq 0$. \qedhere

\end{proof}

\begin{proposition} \label{Hardy_SIS_prop}

Let $f \in H^p(\mathbb{D}^n)$ for some $1 \leq p < \infty$. If there is a sequence $\{f_k\}_{k \in \mathbb{N}} \subset H^{\infty}(\mathbb{D}^n)$ such that $f_k f \xrightarrow{} g$ as $k \xrightarrow{} \infty$ for some $g \in \mathcal{X}$, then $g \in \text{\emph{S}}[f]$. In particular, if there exists a sequence $\{f_k\}_{k \in \mathbb{N}} \subset H^{\infty}(\mathbb{D}^n)$ such that $f_k f \xrightarrow{} g$ for some cyclic $g \in H^p(\mathbb{D}^n)$, then $f$ is cyclic.

\end{proposition}

\begin{proof}

The first part of the proposition follows easily from \emph{Lemma 4.5}, since $f_k f \in \text{S}[f]$ for each $k \in \mathbb{N}$, and S$[f]$ is closed implies $g = \underset{k \xrightarrow{} \infty}{\lim} f_k f \in \text{S}[f]$.

\vspace{1.5mm}

\noindent For the second part, note that $g \in \text{S}[f]$ implies S$[g] \subset \text{S}[f]$. Since $g$ is assumed to be cyclic, S$[g] = H^p(\mathbb{D}^n)$ which means S$[f] = H^p(\mathbb{D}^n)$. Therefore in this case, $f$ is also cyclic. \qedhere

\end{proof}

\noindent The following result follows easily from \emph{Theorem \ref{CPO_general_Q}} and \emph{Proposition \ref{Hardy_SIS_prop}}.

\begin{theorem} \label{CPO_Hardy_Banach}

Suppose $\mathcal{X}$ satisfies properties \textbf{Q1}-\textbf{Q3} over $\hat{D} \subset \mathbb{C}^n$. Let $T : \mathcal{X} \xrightarrow{} H^q(\mathbb{D}^m)$ be a bounded linear map for some $1 \leq q < \infty$. Then, the following are equivalent :

\begin{enumerate}

\item[$(1)$] $T$ preserves cyclicity, i.e. $Tf$ is cyclic in $H^q(\mathbb{D}^m)$ whenever $f$ is cyclic in $\mathcal{X}$.

\item[$(2)$] There exists a cyclic function $a \in H^q(\mathbb{D}^m)$ and an analytic function $b : \mathbb{D}^m \xrightarrow{} \hat{D}$ such that \\ $Tf = a \cdot (f \circ b)$ for every $f \in \mathcal{X}$.
    
\end{enumerate}

\end{theorem}

\begin{proof}

$(1) \Rightarrow (2)$ follows from \emph{Theorem \ref{CPO_general_Q}}.

\vspace{1.5mm}

\noindent For the converse, let $a \in H^q(\mathbb{D}^m)$ and $b : \mathbb{D}^m \xrightarrow{} \hat{D}$ be as in $(2)$. We show that for every cyclic $f \in \mathcal{X}$, $Tf = a \cdot (f \circ b)$ is cyclic in $H^q(\mathbb{D}^m)$.

\vspace{1.5mm}

Since $f$ is cyclic, there exist polynomials $p_k$ such that $p_k f \xrightarrow{} 1$ in $\mathcal{X}$. Since $T$ is a bounded map, $T(p_k f) \xrightarrow{} T(1)$ in $H^q(\mathbb{D}^m)$. Note that $T(1) = a$ is cyclic and $$T(p_k f) = a \cdot (p_k \circ b) \cdot (f \circ b) = (p_k \circ b) \cdot (a \cdot (f \circ b)).$$ It is easy to see $(p_k \circ b) \in H^{\infty}(\mathbb{D}^m)$ for each $n$ since the image of $b$ lies in $\hat{D} \subset \sigma_r(S)$. From the second part of \emph{Proposition \ref{Hardy_SIS_prop}}, since $(p_k \circ b) \cdot (a \cdot (f \circ b)) \xrightarrow{} a$ and $a$ is cyclic, we get that $Tf = a \cdot (f \circ b)$ is cyclic in $H^q(\mathbb{D}^m)$ and thus, $(2) \Rightarrow (1)$. \qedhere

\end{proof}

\begin{remark}

$(i)$ The proof of $(2) \Rightarrow (1)$ relies on \emph{Proposition \ref{Hardy_SIS_prop}}, which further relies on the fact that the dual of $L^p(\mathbb{T}^n)$ for $1 \leq p < \infty$ is $L^{p'}(\mathbb{T}^n)$ where $1/p + 1/p' = 1$ and thus, does not translate easily to other general spaces of analytic functions.

\vspace{1.5mm}

\noindent $(ii)$ Note that the proof of \emph{Theorem \ref{mash_Hardy}} above and \emph{Theorem 2} in \cite{kou2017linear} uses the canonical factorization theorem for Hardy spaces on the unit disc $\mathbb{D}$ (\emph{Theorem 2.8}, \cite{duren1970theory}). We do not have such a result when $n > 1$ (see \emph{Section 4.2} in \cite{rudin1969function}), hence a different approach was needed.

\vspace{1.5mm}

\noindent $(iii)$ Recall that \emph{Theorem \ref{mash_Hardy}} does not require boundedness of $T$ for the proof of $(1) \Rightarrow (2)$ to work when $\mathcal{X} = H^p(\mathbb{D})$. Plus, \emph{Theorem \ref{mash_Hardy}} is valid even for $0 < p < 1$. This is because its proof also depends on the canonical factorization theorem as mentioned above.

\vspace{1.5mm}

\noindent $(iv)$ We will see later in this section that $(1) \Rightarrow (2)$ is still valid for the case when $\mathcal{X} = H^p(\mathbb{D}^n)$ and $\mathcal{Y} = H^q(\mathbb{D}^m)$ for $0 < p,q < 1$ even though they are not Banach spaces. The case $p,q = \infty$ shall be treated separately as well since $H^{\infty}(\mathbb{D}^n)$ is not separable and hence the standard notion of cyclicity does not make any sense.

\end{remark}

We now show that the assumption `$T$ is a bounded operator' can be dropped in a specific case for the Hardy spaces. First, we need the following fact about boundedness of certain composition operators.

\begin{proposition} \label{Composition_boundedness}

For $1 \leq p < \infty$ and an analytic function $b : \mathbb{D}^m \xrightarrow{} \mathbb{D}$, the map $T : H^p(\mathbb{D}) \xrightarrow{} H^p(\mathbb{D}^m)$ defined as $Tf := f \circ b$ is a well-defined bounded linear operator.

\end{proposition}

\begin{proof}

First, we show that $f \circ b \in H^p(\mathbb{D}^m)$ for every $f \in H^p(\mathbb{D})$ which shows $T$ is well-defined. The linearity of $T$ is immediate after that. We use the existence of harmonic majorants for functions in the Hardy spaces and their properties for the rest of the proof. See \emph{Section 3.2} in \cite{rudin1969function} for more details. The argument here is inspired by the one given in the corollary of \emph{Theorem 2.12} in \cite{duren1970theory} for the case $m = 1$.

\vspace{1.5mm}

\noindent Let $U$ be the smallest harmonic majorant of $|f|^p$, i.e. the Poisson integral of $|f(e^{i \theta})|^p$ given by $$ U(r e^{i \theta}) = \frac{1}{2 \pi} \int_0^{2 \pi} P(r, \theta - t) |f(e^{it})|^p dt, \, \text{ where } P(r,\theta) = \text{Re} \bigg ( \frac{1+r e^{i \theta}}{1-r e^{i \theta}} \bigg ).$$ Then, $|f(u)|^p \leq U(u)$ for all $u \in \mathbb{D}$ which implies that $|Tf(z)|^p \leq U(b(z))$ for every $z \in \mathbb{D}^m$.

\vspace{1.5mm}

\noindent Since $U$ is harmonic, $U =$ Re$(g)$ for some analytic function $g : \mathbb{D} \xrightarrow{} \mathbb{C}$. This means that $U \circ b =$ Re$(g \circ b)$ which implies that $U \circ b$ is an $m$-harmonic function and thus, a harmonic majorant for $|Tf|^p = |f \circ b|^p$. This proves that $f \circ b \in H^p(\mathbb{D}^m)$ and so, $T$ is a well-defined map. To show $T$ is bounded, observe that $$ M_p(r,f \circ b)^p \leq U(b(0)) \leq \frac{1 + |b(0)|}{1 - |b(0)|} ||f||^p, \, \text{ where } M_p(r,f \circ b) := \Big ( \underset{r \mathbb{T}^m}{\int} |f \circ b|^p d \sigma_m \Big )^{\frac{1}{p}}. $$ The first inequality follows from the mean value property of $m$-harmonic functions, and the second inequality follows from the fact that $P(r , \theta) \leq (1+r)/(1-r)$ for all values of $r$ and $\theta$. Taking supremum over $r$ in the above inequality, we get $$||f \circ b|| \leq \Big ( \frac{1 + |b(0)|}{1 - |b(0)|} \Big )^{1/p} ||f||$$ for every $f \in H^p(\mathbb{D})$ and thus, $T$ is bounded. \qedhere

\end{proof}

The following result follows easily from the previous discussion.

\begin{proposition} \label{CPO_Hardy_composition}

Fix $1 \leq p < \infty$ and let $T : H^p(\mathbb{D}) \xrightarrow{} H^p(\mathbb{D}^m)$ be a linear map such that $T1 = 1$. Then, the following are equivalent :

\begin{enumerate} 

\item[$(1)$] $T$ is a bounded linear map that preserves cyclicity

\item[$(2)$] There exists an analytic function $b : \mathbb{D}^m \xrightarrow{} \mathbb{D}$ such that $Tf = f \circ b$ for every $f \in H^p(\mathbb{D})$

\end{enumerate}

\end{proposition}

\begin{proof}

As before, $(1) \Rightarrow (2)$ follows directly from \emph{Theorem 4.3}.

\vspace{1.5mm}

\noindent For the converse, let $b : \mathbb{D}^m \xrightarrow{} \mathbb{D}$ be an analytic function such that $Tf = f \circ b$ for each $f \in H^p(\mathbb{D})$. By \emph{Proposition \ref{Composition_boundedness}}, $T$ is a bounded linear operator. By $(2) \Rightarrow (1)$ in \emph{Theorem \ref{CPO_Hardy_Banach}}, $T$ preserves cyclicity. \qedhere

\end{proof}

\begin{remark}

Note that the only place we use that the domain of $H^p(\mathbb{D})$ is in one variable, is to show boundedness of $f \mapsto f \circ b$ for every $b : \mathbb{D}^m \xrightarrow{} \mathbb{D}$. More precisely, we use the fact that any harmonic function $U$ in one variable is the real part of some holomorphic function. This is not true for $n > 1$ (see \emph{Section 2.4} in \cite{rudin1969function}).

\end{remark}

As mentioned in \emph{Remark $(iv)$} under \emph{Theorem \ref{CPO_Hardy_Banach}}, we now consider the cases $0 < p < 1$ and $p = \infty$.

\begin{example} \label{example_p<1}

First, we address the case $\mathcal{X} = H^{p}(\mathbb{D}^n)$ for $0 < p < 1$. Fix $0 < p <1$ and note that $H^p(\mathbb{D}^n)$ satisfies \textbf{P1}-\textbf{P3} if we replace boundedness with continuity. The issue is that $H^p(\mathbb{D}^n)$ is not a Banach space.

\vspace{1.5mm}

\noindent Even though $H^p(\mathbb{D}^n)$ is not normable, it is still a complete metric space under $d_p(f,g) := ||f-g||_p^p$ where $|| \cdot ||_p$ is as defined in \emph{Section 1}.

\vspace{1.5mm}

Using this, its bounded linear functionals can be defined in the usual manner. That is, we say that $\Lambda : H^p(\mathbb{D}^n) \xrightarrow{} \mathbb{C}$ is bounded if $$ ||\Lambda|| := \underset{||f||_p = 1}{\text{sup}} |\Lambda(f)| < \infty $$ This means that $|\Lambda(f)| \leq ||\Lambda|| \cdot ||f||$ for all bounded $\Lambda$, and $f \in H^p(\mathbb{D}^n)$. It is easy to verify that this notion of boundedness is equivalent to the continuity of $\Lambda$. Similarly, we say an operator $T : H^p(\mathbb{D}^n) \xrightarrow{} H^q(\mathbb{D}^m)$ for some $0 < q \leq \infty$ is bounded if \[ ||T|| := \underset{||f||_p = 1}{\text{sup}} ||Tf||_q < \infty. \] As was the case with linear functionals, it is easy to verify that this notion of boundedness is equivalent to the continuity of $T$. This implies that for every bounded linear operator $T$ on $H^p(\mathbb{D}^n)$ and $f \in H^p(\mathbb{D}^n)$, $||Tf|| \leq ||T|| \cdot ||f||$. This means \emph{Lemmas \ref{exponential_existence}} and \emph{4.2} hold for $\mathcal{X} = H^p(\mathbb{D}^n)$ even when $0 < p < 1$.

\vspace{1.5mm}

\noindent In order to show that \emph{Theorem \ref{CPO_general_Q}} holds for $\mathcal{X} = H^p(\mathbb{D}^n)$, we only need to show that \emph{Theorem \ref{fundamental_Q}} holds since the arguments in the proof of \emph{Theorem \ref{CPO_general_Q}} do not rely on the Banach space structure of $\mathcal{X}$ except when \emph{Theorem \ref{fundamental_Q}} is applied. First, we show that the maximal domain for functions in $H^p(\mathbb{D}^n)$ when $0 < p < 1$ is also $\mathbb{D}^n$.

\vspace{1.5mm}

\noindent Fix $0 < p < 1$. We will show as in \emph{Example \ref{example_Hardy_envelope}} that the family $\mathcal{F} := \{ z_i - \beta \; \big \rvert \; 1 \leq i \leq n \text{ and } \beta \not \in \mathbb{D} \}$ is an envelope of cyclic polynomials in $H^p(\mathbb{D}^n)$. Let $b \in \overline{\mathbb{D}^n}$ be such that $\Lambda_b \rvert _{\mathcal{P}}$ extends to a bounded linear functional $ \Lambda \in H^p(\mathbb{D}^n)$. Thus $b_j \in \mathbb{T}$ for some $1 \leq j \leq n$.

\vspace{1.5mm}

\noindent Using subordination and an argument similar to \emph{Theorem 3.2} in \cite{duren1970theory}, we can show that $1/(z - b_j) \in H^p(\mathbb{D})$. Consider a sequence of polynomials $\{q_k\}_{k \in \mathbb{N}}$ such that $q_k \xrightarrow{} 1/(z - b_j)$ in $H^p(\mathbb{D})$ and note that $q_k (z - b_j) \xrightarrow{} 1$ since multiplication by $z - b_j$ is a bounded linear operator on $H^p(\mathbb{D})$. This implies $z - b_j$ is cyclic in $H^p(\mathbb{D})$ and thus, $q(z) := z_j - b_j$ is cyclic in $H^p(\mathbb{D}^n)$. Clearly $\mathcal{F}$ defined above is then an envelope of cyclic polynomials. This means that for any given $f \in H^p(\mathbb{D}^n)$, there exists a sequence of polynomials $\{p_k\}_{k \in \mathbb{N}}$ such that $p_k q \xrightarrow{} f$. Since $q(b) = 0$, we get $$\Lambda(f) = \underset{k \xrightarrow{} \infty}{\lim} \Lambda(p_k q) = \underset{k \xrightarrow{} \infty}{\lim} p_k(b) q(b) = 0.$$ This means $\Lambda \equiv 0$, a contradiction. So, $b \in \mathbb{D}^n$ and we get $\hat{D} = \mathbb{D}^n$ for $H^p(\mathbb{D}^n)$ when $0 < p < 1$.

\vspace{1.5mm}

It is known that all outer functions in $H^p(\mathbb{D})$ are also cyclic for $0 < p < 1$ (\emph{Theorem 4}, \cite{10.2307/1994442}) but we prove $z - b_j$ is cyclic differently to show one way of obtaining an envelope of cyclic polynomials in the case when $\mathcal{X}$ is not a Banach space. Notice that the only other place we use the norm in the proof of \emph{Theorem \ref{fundamental_P}} (and hence \emph{Theorem \ref{fundamental_Q}}) is to obtain the non-vanishing entire function $F(w) := \underset{\alpha \in \mathbb{Z}^+(n)}{\sum} \frac{\Lambda(z^{\alpha}) \cdot w^{\alpha}}{\alpha !}$ using $$|\Lambda(z^{\alpha})| \leq ||\Lambda|| \cdot ||z^{\alpha}|| \leq ||\Lambda|| \cdot ||S_1||^{\alpha_1} \dots||S_n||^{\alpha_n} \cdot ||1||, \, \text{for every } \alpha \in \mathbb{Z}^+(n).$$ As we saw above, this should not be an issue for $H^p(\mathbb{D}^n)$ since $||\Lambda||$ makes just as much sense and $||z^{\alpha}|| = 1$ for all $\alpha \in \mathbb{Z}^+(n)$. This gives us $|\Lambda(z^{\alpha})| \leq ||\Lambda||$, which is good enough for the rest of the proof to work. Therefore \emph{Theorem \ref{fundamental_Q}} holds for $\mathcal{X} = H^p(\mathbb{D}^n)$ and $\mathcal{Y} = H^q(\mathbb{D}^m)$, and so does \emph{Theorem \ref{CPO_general_Q}} even when $0 < p,q < 1$.

\end{example}

\begin{example} \label{example_p_infty}

The case $p = \infty$ is a little different since the problem here is that the set of polynomials is not dense in $H^{\infty}(\mathbb{D}^n)$. In fact $H^{\infty}(\mathbb{D}^n)$ is not separable so cyclicity of functions does not make sense. However in the case when $n = 1$, we know that all cyclic functions in $H^p(\mathbb{D})$ for $0 < p < \infty$ are \emph{outer functions} and vice versa (\emph{Theorem 7.4}, \cite{duren1970theory}, \emph{Theorem 4}, \cite{10.2307/1994442}). Since outer functions do make sense for $n \geq 1$ and $p = \infty$, we can talk about outer functions instead of cyclic functions in this case.

\begin{definition} \label{definition_outer}

For $0 < p \leq \infty$, a function $f \in H^p(\mathbb{D}^n)$ is said to be outer in $H^p(\mathbb{D}^n)$ if $\log|f(0)| = \int_{\mathbb{T}^n} \log|f|$.

\end{definition}

\noindent It is known that the class of outer functions is different from cyclic functions when $n > 1$. In fact, all cyclic functions are outer, but there are outer functions that are not cyclic. The proof of this fact, and more details on outer functions can be found in \emph{Section 4.4} of \cite{rudin1969function}. The argument uses different composition operators between $H^2(\mathbb{D}^2)$ and a few different Hardy spaces on the unit disc to explicitly find such a function. This was a motivating result for studying operators that preserve outer/cyclic functions with the hope of getting a better understanding of the difference between outer and cyclic functions. We discuss this a little bit in \emph{Theorem \ref{outer_but_not_cyclic_example}} below.

\vspace{1.5mm}

Note that in the case $p = \infty$, the hypothesis of \emph{Theorem \ref{fundamental_Q}} does not make sense. In fact, we will completely avoid using maximal domains for $H^{\infty}(\mathbb{D}^n)$ since without cyclicity, we cannot even determine if $\hat{D} \subset \sigma_r(S)$. Instead, consider $\Lambda \in \big (H^{\infty}(\mathbb{D}^n) \big )^*$ such that $\Lambda(f) \neq 0$ for all outer functions $f \in H^{\infty}(\mathbb{D}^n)$. Since $e^{w \cdot z}$ is clearly an outer function for all $w \in \mathbb{C}^n$, we proceed as in the proof of \emph{Theorem \ref{fundamental_P}} to obtain $\Lambda |_{\mathcal{P}} \equiv a \Lambda_b |_{\mathcal{P}}$ for some $a \in \mathbb{C} \setminus \{ 0 \}$ and $b \in \mathbb{C}^n$.

\vspace{1.5mm}

\noindent Now, proceed as in \emph{Example \ref{example_Hardy_envelope}} and instead of having an envelope of cyclic polynomials, we now have an envelope of outer polynomials which is the same set of polynomials $\mathcal{F} = \{ z_i - \beta \; | \; 1 \leq i \leq n, \, \beta \not \in \mathbb{D} \} $. Since $\Lambda(f) \neq 0$ for all outer functions $f$, we get $b_i - \beta \neq 0$ for all $1 \leq i \leq n$ and $\beta \not \in \mathbb{D}$ which implies $b_i \in \mathbb{D}$ for every $1 \leq i \leq n$. Therefore, $b \in \mathbb{D}^n$ and $\Lambda f = a \cdot f(b)$ for all $f \in H^{\infty}(\mathbb{D}^n)$.

\vspace{1.5mm}

\noindent Thus, the conclusion of \emph{Theorem \ref{fundamental_P}} is valid for $H^{\infty}(\mathbb{D}^n)$ if we consider $\Lambda \in \big ( H^{\infty}(\mathbb{D}^n) \big )^*$ that acts on outer functions as above and so, \emph{Theorem \ref{CPO_general_P}} is valid for $\mathcal{X} = H^{\infty}(\mathbb{D}^n)$ if we replace cyclic functions with outer functions and $\hat{D}$ with $\mathbb{D}^n$. A similar logic can be applied to operators that preserve outer functions in $H^p(\mathbb{D}^n)$ for $0 < p < \infty$.

\end{example}

\noindent This discussion about Hardy spaces above yields the following generalization of \emph{Theorem \ref{mash_Hardy}} for bounded operators from $H^p(\mathbb{D}^n)$ into $H^q(\mathbb{D}^m)$ for all $0 < p,q \leq \infty$ and $m,n \in \mathbb{N}$ that preserve cyclic/outer functions.

\begin{theorem} \label{CPO_Hardy_full}

$(1)$ Fix $0 < p,q < \infty$ and $m, n \in \mathbb{N}$. Let $T : H^p(\mathbb{D}^n) \xrightarrow{} H^q(\mathbb{D}^m)$ be a bounded linear operator such that $Tf$ is cyclic whenever $f$ is cyclic. Then, $T$ is necessarily a weighted composition operator, i.e. there exist analytic functions $a \in H^q(\mathbb{D}^m)$ and $b : \mathbb{D}^m \xrightarrow{} \mathbb{D}^n$ such that $Tf(z) = a(z) f(b(z))$ for every $z \in \mathbb{D}^m$ and $f \in H^p(\mathbb{D}^n)$.

\vspace{1.5mm}

Furthermore $a = T1$ is cyclic and $b = \frac{T(z)}{T1}$, where $T(z) = \big (T(z_i) \big )_{i = 1}^n$.

\vspace{1.5mm}

\noindent In the case when $1 \leq q < \infty$, the converse is also true. That is, all bounded weighted composition operators from $H^p(\mathbb{D}^n)$ into $H^q(\mathbb{D}^m)$ also preserve cyclicity.

\vspace{1.5mm}

$(2)$ Fix $0 < p,q \leq \infty$ and $m,n \in \mathbb{N}$. Let $T : H^{p}(\mathbb{D}^n) \xrightarrow{} H^{q}(\mathbb{D}^m)$ be a bounded linear operator such that $Tf$ is outer whenever $f$ is outer. Then, $T$ is necessarily a weighted composition operator, i.e. there exist analytic functions $a \in H^q(\mathbb{D}^m)$ and $b : \mathbb{D}^m \xrightarrow{} \mathbb{D}^n$ such that $Tf(z) = a(z) f(b(z))$ for every $z \in \mathbb{D}^m$ and $f \in H^{p}(\mathbb{D}^n)$.

\vspace{1.5mm}

Furthermore, $a = T1$ is outer and $b = \frac{T(z)}{T1}$, where $T(z) = \big (T(z_i) \big )_{i = 1}^n$.

\end{theorem}

\begin{remark}

$(i)$ Note that the proof of \emph{Proposition \ref{Hardy_SIS_prop}} above is not valid for $0 < q < 1$ or $q = \infty$ since we use the duality of $L^q(\mathbb{T}^m)$ when $1 \leq q < \infty$. Therefore we do not obtain a result like \emph{Theorem \ref{CPO_Hardy_Banach}} when $0 < q < 1$ or $q = \infty$. \emph{Theorem \ref{CPO_Hardy_full}} is probably the best we can expect in these cases with our techniques.

\vspace{1.5mm}

\noindent $(ii)$ It is not easy to determine when even a composition operator would preserve outer functions. One can check that in certain cases (like part $(b)$ of \emph{Lemma 4.4.4} in \cite{rudin1969function}) the map $f \mapsto f \circ b$ preserves outer functions but more generally, it is not known to be true.

\vspace{1.5mm}

\noindent It would be interesting to characterize all weighted composition operators that preserve outer functions since it might help us understand the difference between outer functions and cyclic functions in $H^p(\mathbb{D}^n)$ for $n > 1$. Unless, of course, all bounded weighted composition operators also preserve outer functions in which case we do not obtain any distinguishing property using cyclicity/outer preserving operators but rather, a kind of `\emph{linear rigidity}' between outer and cyclic functions. The following result shows that such a rigidity does not exist in general when $n > 1$.

\end{remark}

\begin{theorem} \label{outer_but_not_cyclic_example}

Let $0 < q < 1/2$ be arbitrary. There exists a bounded linear map $T : H^2(\mathbb{D}^2) \xrightarrow{} H^q(\mathbb{D})$ such that it preserves cyclicity, but not outer functions.

\end{theorem}

\begin{proof}

This example is from \cite{rudin1969function} but it was used in a different context; to obtain an outer function in $H^2(\mathbb{D}^2)$ which is not cyclic. We refer the reader to the discussion surrounding \emph{Theorem 4.4.8} in \cite{rudin1969function} for proofs of the facts mentioned below.

\vspace{1.5mm}

\noindent Fix $0 < q < 1/2$. Let $T : H^2(\mathbb{D}^2) \xrightarrow{} H^q(\mathbb{D})$ be defined as $Tf(z) = f \big ( \frac{1+z}{2},\frac{1+z}{2} \big )$ for every $z \in \mathbb{D}$ and $f \in H^2(\mathbb{D}^2)$. Then $T$ is a bounded linear operator that preserves cyclicity. The function $f(z_1,z_2) = e^{\frac{z_1 + z_2 + 2}{z_1 + z_2 - 2}}$ is outer in $H^2(\mathbb{D}^2)$ but $Tf(z) = e^{\frac{z+3}{z-1}}$ is not outer in $H^q(\mathbb{D})$ and so, $T$ does not preserve outer functions. \qedhere

\end{proof}

\noindent $(iii)$ Notice that the proof of $(1) \Rightarrow (2)$ in \emph{Theorem \ref{CPO_general_Q}} depends mostly on the properties of $\mathcal{X}$, since $\mathcal{Y}$ can be chosen to be fairly general. On the other hand, all the discussion about Hardy spaces shows that the proof of $(2) \Rightarrow (1)$ in \emph{Theorem \ref{CPO_general_Q}} depends on the properties of $\mathcal{Y}$. In \emph{Proposition \ref{Hardy_SIS_prop}}, we saw that the proof relies heavily on the properties of $H^p(\mathbb{D}^n)$ and might not work for other spaces. This shows that while properties \textbf{Q1}-\textbf{Q3} are quite reasonable for $\mathcal{X}$, it is not completely obvious what properties $\mathcal{Y}$ needs to have generally in order for the converse of \emph{Theorem \ref{CPO_general_Q}} to hold true.

\vspace{1.5mm}

To show some different application of the abstract results proved in \emph{Sections 2} and \emph{3}, we conclude our discussion by proving a GK\.Z-type theorem for spaces of analytic functions.

\section{GK\.Z-type theorem for spaces of analytic functions}

The following result was proved independently by A. M. Gleason (\emph{Theorem 1}, \cite{gleason1967characterization}), and J.-P. Kahane and W. \. Zelazko (\emph{Theorem 1}, \cite{kahane1968characterization}).

\begin{theorem} \label{GKZ}

Let $\mathcal{B}$ be a complex unital Banach algebra, and let $\Lambda \in \mathcal{B}^*$ be such that $\Lambda(1) = 1$. Then, $\Lambda(ab) = \Lambda(a) \Lambda(b)$ for every $a,b \in \mathcal{B}$ if and only if $\Lambda(a) \neq 0$ for every $a$ invertible in $\mathcal{B}$.

\end{theorem}

\noindent We shall prove a similar result about \emph{partially multiplicative linear functionals} on spaces of analytic functions as an interesting byproduct of topics discussed in \emph{Sections 2} and \emph{3}.

\begin{definition} \label{definition_partially_multiplicative}

Suppose $\mathcal{X}$ is a space of functions that satisfies \textbf{Q1}-\textbf{Q3} over $\hat{D} \subset \mathbb{C}^n$. We will consider two types of partially multiplicative linear functionals $\Lambda \in \mathcal{X}^*$ as follows.

\begin{enumerate}

\item[\textbf{M1}] For every $\phi \in \mathcal{M}(\mathcal{X})$, $f \in \mathcal{X}$ we have $\Lambda(\phi f) = \Lambda(\phi) \Lambda(f)$.

\item[\textbf{M2}] For every $f,g \in \mathcal{X}$ such that $fg \in \mathcal{X}$ we have $\Lambda(fg) = \Lambda(f) \Lambda(g)$.

\end{enumerate}

\end{definition}

\noindent Note that if $\Lambda$ is \textbf{M2}, then $\Lambda$ is also \textbf{M1}. The converse need not be true in general. With this notation, the GK\.Z-type theorem for spaces $\mathcal{X}$ that satisfies \textbf{Q1}-\textbf{Q3} is as follows.

\begin{theorem} \label{GKZ_general}

Suppose $\widehat{\mathcal{X}}$ is a space of analytic functions on a set $\widehat{D} \subset \mathbb{C}^n$ that satisfies \textbf{Q1}-\textbf{Q3}. Let $\Lambda \in \mathcal{X}^*$ such that $\Lambda(1) = 1$. Then, the following are equivalent :

\begin{enumerate}

\item[$(i)$] $\Lambda(e^{w \cdot z}) \neq 0$ for every $w \in \mathbb{C}^n$.

\item[$(ii)$] $\Lambda \equiv \Lambda_b$ for some $b \in \widehat{D}$.

\item[$(iii)$] $\Lambda$ is \textbf{M2}.

\item[$(iv)$] $\Lambda$ is \textbf{M1}.

\end{enumerate}

\end{theorem}

\begin{proof}

Note that $(i) \Rightarrow (ii)$ follows from \emph{Theorem \ref{fundamental_Q}}, and that $(II) \Rightarrow (iii) \Rightarrow (iv)$ is obvious from the definitions of \textbf{M1} and \textbf{M2} functionals.

\vspace{2mm}

\noindent For the proof of $(iv) \Rightarrow (i)$, assume $\Lambda$ is \textbf{M1} and note that $e^{w \cdot z} \in \mathcal{M}(\widehat{\mathcal{X}})$ for every $w \in \mathbb{C}^n$. Thus, $$\Lambda(e^{w \cdot z}) \Lambda(e^{-w \cdot z}) = \Lambda(e^{w \cdot z} \cdot e^{-w \cdot z}) = \Lambda(1) = 1$$ for every $w \in \mathbb{C}^n$. Therefore $\Lambda(e^{w \cdot z}) \neq 0$ for every $w \in \mathbb{C}^n$. \qedhere

\end{proof}

We now show that the notion of maximal domains that we used in this paper comes from a class of partially multiplicative linear functionals.

\begin{definition}

Suppose $\mathcal{X}$ satisfies \textbf{P1}-\textbf{P3}. Let $\Lambda$ be a linear functional. We say $\Lambda$ is \textbf{M0} if for every choice of polynomials $p$ and $q$, $\Lambda(pq) = \Lambda(p) \Lambda(q)$. 

\end{definition}

\noindent With this definition in mind, we have the following simple characterization of \textbf{M0} functionals.

\begin{proposition}

Suppose $\mathcal{X}$ satisfies \textbf{P1}-\textbf{P3} over an open set $D \subset \mathbb{C}^n$. Then, $\Lambda \in \mathcal{X}^*$ is \textbf{M0} if and only if $\Lambda |_{\mathcal{P}} \equiv \Lambda_b |_{\mathcal{P}}$ for some $b \in \sigma_r(S)$.

\end{proposition}

\begin{proof}

Suppose $\Lambda \in \mathcal{X}^*$ is \textbf{M0}. This means that for every $1 \leq i \leq n$ and $k \in \mathbb{N}$, $\Lambda(z_i^k) = (\Lambda(z_i))^k$. Pick $b = (\Lambda(z_i))_{i = 1}^n \in \mathbb{C}^n$ so that $\Lambda(p) = p(b)$ for every polynomial $p$. To show $b$ lies in $\sigma_r(S)$ is the same as that in the proof of \emph{Theorem \ref{fundamental_P}}.

\vspace{1.5mm}

\noindent The converse is obvious since for every choice of polynomials $p$ and $q$, $\Lambda(pq) = p(b)q(b) = \Lambda(p) \Lambda(q)$. \qedhere

\end{proof}

\noindent What this means is that all reasonable notions of partially multiplicative linear functionals align when we consider these nice spaces of analytic functions, and \emph{Theorem \ref{GKZ_general}} gives a way to describe these functionals in terms of their action on a certain set of exponential functions.

\vspace{1.5mm}

A similar result about different notions of partially multiplicative linear functionals on reproducing kernel Hilbert spaces with the complete Pick property was recently discovered (\emph{Corollary 3.4}, \cite{aleman2017smirnov}). It was proved that in the case of a complete Pick space, \textbf{M1} and \textbf{M2} are equivalent. It should be noted that this is not a special case of \emph{Theorem \ref{GKZ_general}} since it covers Hilbert spaces of functions that are not necessarily analytic. On the other hand, \emph{Theorem \ref{GKZ_general}} covers certain Banach spaces of analytic functions and not just Hilbert spaces.

\vspace{1.5mm}

\noindent It is also worth mentioning that just as we devised a maximal domain from \textbf{M0} functionals on $\mathcal{X}$, one can construct a different notion of maximal domain from \textbf{M1} and \textbf{M2}. Depending on what properties we want the extension $\hat{\mathcal{X}}$ to have, we may want to choose between \textbf{M0}, \textbf{M1} and \textbf{M2}.

\vspace{1.5mm}

Since we were interested in preserving the behaviour of the shift operators and cyclic functions, \textbf{M0} seemed appropriate as it respects multiplication of any function in $\mathcal{X}$ with any polynomial. If we were interested in a more algebraic behaviour as in \cite{aleman2017smirnov}, we would choose \textbf{M2} instead since it is the strongest notion of partially multiplicative linear functionals we can expect from spaces of functions that are not algebras already. For a detailed discussion on this topic, refer to \emph{Section 2} in \cite{mccarthy2017spaces}.

\section{Appendix}

In the proof of \emph{Theorem \ref{fundamental_P}}, we used it as a fact that all non-vanishing entire functions that satisfy a nice growth condition are of a specific type. This is not entirely obvious and we give a proof for it in this section. Since the proof of this fact is independent of all the other discussion, it seemed appropriate to present it separately.

\begin{theorem} \label{non_vanishing_entire}

Fix $n \in \mathbb{N}$. Let $F \in \text{\emph{Hol}}(\mathbb{C}^n)$ be a non-vanishing entire function for which there exist constants $A, B$ such that $|F(z)| \leq A e^{B r^m}$ for all $z \in (r \mathbb{D})^n$ and for all $r > 0$. Then, there exists a polynomial $p$ with $\text{\emph{deg}}(p) \leq m$ such that $F(z) = e^{p(z)}$ for all $z \in \mathbb{C}^n$.

\end{theorem}

\begin{proof}

Since $F$ is non-vanishing, there exists an entire function $G$ such that $F = e^G$. Note that the hypothesis then implies $\text{Re}(G) \leq \ln{A} + Br^m$ in $(r \mathbb{D})^n$. Let $C := \text{max}\{\ln{A},B\}+1$ so that $\text{Re}(G) < C(1 + r^m)$ in $(r \mathbb{D}^n)$. We need to show that this implies $G$ is a polynomial with deg$(G) \leq m$.

\vspace{1.5mm}

\noindent First, consider the case $n = 1$. Even though this is a well-known result, we provide a short proof for the sake of completion. Fix $r > 0$ and define $g_r(z) := G(rz)$ for all $z \in \mathbb{C}$. By the hypothesis, $g_r$ maps the unit disc $\mathbb{D}$ into $\{ \text{Re}(w) < C(1 + r^m) \}$ which is a left half-plane.

\vspace{1.5mm}

\noindent Using the M\"obius transformation $\phi(z) := C(1 + r^m) + \frac{z+1}{z-1}$ which maps $\mathbb{D}$ onto $\{ \text{Re}(w) < C(1 + r^m) \}$, and applying Schwarz lemma to $\phi^{-1} \circ g_r$, we get \[ \Bigg \rvert \frac{g_r(z)-g_r(w)}{2C(1+r^m)-g_r(z)-\overline{g_r(w)}} \Bigg \rvert \leq \bigg \rvert \frac{z-w}{1 - z \overline{w}} \bigg \rvert \] After plugging in $w = 0$, clearing out the denominators, and using triangle inequality we get \[ |g_r(z) - G(0)| \leq \big ( 2C(1+r^m) + |g_r(z)| + |G(0)| \big ) \cdot |z| \text{ for every } z \in \mathbb{D} \] Fixing $|z| = 1/2$, we get $|g_r(z)| - |G(0)| \leq |g_r(z) - G(0)| \leq C(1+r^m) + \frac{|g_r(z)|}{2} + \frac{|G(0)|}{2}$ which simplifies to \[ |G(rz)| \leq 2C(1+r^m) + 3|G(0)| \text{ whenever } |z| = 1/2 \] Since the choice of $r > 0$ was arbitrary, we get that for every $z \in \mathbb{C}$, $$|G(z)| \leq 2C(1+2^m |z|^m) + 3|G(0)|.$$ Choosing another appropriate constant $C_0$, we get $|G(z)| \leq C_0(1 + r^m)$ for all $r > |z|$. If $G = \underset{k \in \mathbb{N}}{\sum} a_k z^k$ is the power-series of $G$, then using the above relation with Cauchy estimates for $a_k$ gives us $|a_k| \leq C_1/r^{k - m} $ for some constant $C_1$ and all $r > 0$. Thus, $a_k = 0$ for all $k > m$ which means $G$ is a polynomial with deg$(G) \leq m$ as desired.

\vspace{1.4mm}

Assume now that $n > 1$. Let $G(z) = \underset{k \in \mathbb{N}}{\sum} G_k(z)$ be the homogeneous expansion of $G$. Fix $z \in \mathbb{C}^n$ and let $g_z(\lambda) := G(\lambda z) = \underset{k \in \mathbb{N}}{\sum} \lambda^k G_k(z)$ for $\lambda \in \mathbb{C}$. Notice that $$\text{Re}(g_z(\lambda)) = \text{Re}(G(\lambda z)) \leq \ln{A} + B \cdot C^m |\lambda|^m$$ where $C = \underset{1 \leq j \leq n}{\text{sup}} |z_j|$, since $z \in (r \mathbb{D})^n$ for every $r > C$. Thus, $$\text{Re}(g_z(\lambda)) \leq \ln{A} + B \cdot C^m r^m$$ whenever $\lambda \in r \mathbb{D}$.

\vspace{1.5mm}

\noindent By the one variable case, we know that this implies $G_k(z) = 0$ for all $k > m$. Since the choice of $z \in \mathbb{C}^n$ was arbitrary, this means $G_k(z) = 0$ for all $z \in \mathbb{C}^n$ and $k > m$. Therefore $G$ is a polynomial with deg$(G) \leq m$. \qedhere

\end{proof}

\bibliographystyle{plain}

\bibliography{main}

\end{document}